\documentclass[11pt,letterpaper,reqno]{amsart}
\usepackage{fullpage}
\usepackage{amsmath,amsthm,amssymb,amscd}
\usepackage{enumerate}
\usepackage{enumitem}
\usepackage{array}
\usepackage{float}
\usepackage{bbm}
\usepackage{bm}
\usepackage{stmaryrd}
\usepackage{comment}
\usepackage{mathtools}
\usepackage{dsfont}
\usepackage{xcolor}

\usepackage{hyperref}
\hypersetup{
    colorlinks=true,
    linkcolor=blue,
    citecolor=blue,
    urlcolor=blue
}

\newtheorem{theorem}{Theorem}[section]

\newtheorem{corollary}[theorem]{Corollary}
\newtheorem{lemma}[theorem]{Lemma}

\newtheorem{proposition}[theorem]{Proposition}
\newtheorem{conjecture}[theorem]{Conjecture}

\theoremstyle{definition}
\newtheorem{definition}[theorem]{Definition}

\newtheorem{remark}[theorem]{Remark}

\mathtoolsset{showonlyrefs=true}
\numberwithin{equation}{section}
\numberwithin{figure}{section}
\numberwithin{table}{section}

\allowdisplaybreaks

\let\subsectiontemp\subsection
\renewcommand{\subsection}[1]{ 
    \subsectiontemp{#1} \hfill\vspace{0.5\linespacing} 
}

\newcommand{\lrabs}[1]{\!\left\lvert #1 \right\lvert}
\newcommand{\lrp}[1]{\!\left(#1\right)}
\newcommand{\lrb}[1]{\!\left[#1\right]}
\newcommand{\lrcb}[1]{\!\left\{#1\right\}}

\newcommand{\ZZ}{\mathbb{Z}}

\newcommand{\QQ}{\mathbb{Q}}

\newcommand{\nw}{\mathrm{new}}
\newcommand{\eigen}{\operatorname{eigen}}

\renewcommand{\Im}{\mathrm{Im}}

\title{Only finitely many modular Jacobians are \\ supersingular modulo a given prime}

\author[A. J. Kumar]{Aarya J. Kumar}
\address[A. J. Kumar]{Princeton University}
\email{ajkumar@princeton.edu}

\author[S. Mondal]{Sargam Mondal}
\address[S. Mondal]{Massachusetts Institute of Technology}
\email{sargam@mit.edu}

\author[E. Ross]{Erick Ross}
\address[E. Ross]{School of Mathematical and Statistical Sciences, Clemson University, Clemson, SC}

\author[H. Xue]{Hui Xue}
\address[H. Xue]{School of Mathematical and Statistical Sciences, Clemson University, Clemson, SC}
\email{huixue@clemson.edu}

\begin{document}

\begin{abstract}
In this paper, we study supersingularity of modular Jacobians $J_0(N)$ and abelian varieties $A_f$ of $\mathrm{GL}_{2}$-type from both the vertical perspective (where the prime $p$ is fixed, and the level $N$ varies) and the horizontal perspective (where the newform $f$ is fixed, and the prime $p$ varies). Vertically, we prove that $J_{0}(N)$ is supersingular modulo a given prime $p$ for only finitely many levels $N$. Horizontally, we show that for sufficiently large $p$, the reduction of $A_f$ modulo $p$ is supersingular if and only if it isogenous to a power of a supersingular elliptic curve.
\end{abstract}

\maketitle

\begin{center}
    \textit{In memory of our coauthor and friend, Erick Ross, who tragically passed away during the preparation of this work.}
\end{center}
\vspace{1ex}

\section{Introduction}\label{section:introduction}

Let $A$ be an abelian variety of dimension $g$ defined over $\mathbb{Q}$ and $p$ be a prime of good reduction for $A$. Let $A/\mathbb{F}_p$ denote the reduction of $A$ at $p$, and let $P_{A,p}(X) \in \mathbb{Z}[X]$ denote the Weil polynomial of $A/\mathbb{F}_p$, so that the roots of $P_{A,p}(X)$ are the Frobenius eigenvalues of $A/\mathbb{F}_p$, each of magnitude $\sqrt{p}$ by the Weil conjectures. We say that $A$ is \emph{supersingular} at $p$ if $A/\mathbb{F}_p$ is isogenous over $\overline{\mathbb{F}}_p$ to a product of supersingular elliptic curves. Equivalently, $A$ is supersingular at $p$ if every root of $P_{A,p}(X)$ has $p$-adic valuation $1/2$, so that the Newton polygon of $A_p$ is the straight line of slope $1/2$. By the Manin--Oort criterion (see \cite[Section 1]{Ayotte2016}, for example), this is in turn equivalent to the condition that every normalized Frobenius eigenvalue is a root of unity. 

An interesting problem is to determine the distribution of supersingular primes for a given abelian variety $A/\mathbb{Q}$. Let us set the following notation.
\begin{definition}
    For an abelian variety $A/\mathbb{Q},$ let
    $\operatorname{SS}(A) := \left\{ \text{$p$ prime} : \text{$A$ is supersingular modulo $p$}  \right\}\!.$
\end{definition}
Elkies \cite{elkies-1987} showed that $\operatorname{SS}(E)$ is infinite for any elliptic curve $E/\mathbb{Q}.$ For higher-dimensional abelian varieties, the analogous statements are largely conjectural: it is expected, but not known in general, that $ \operatorname{SS}(A)$ is finite for generic non-CM abelian varieties of dimension $g > 1$ over $\mathbb{Q}.$ Hui \cite{HuiSS2025} recently proved that $\operatorname{SS}(A)$ has density zero for $A/\mathbb{Q}$ a non-CM abelian variety or K3 surface. On the other hand, the infinitude of $\operatorname{SS}(A)$ has been established in some special cases: e.g.,  for certain abelian surfaces with quaternionic multiplication \cite{baba-2006}, and for certain abelian fourfolds \cite{4folds2025}. However, the vertical question -- for fixed $p$, how supersingularity behaves as $A$ varies over a family -- has received comparatively little attention. In this paper, we study both the vertical and horizontal perspectives of this question in the setting of modular Jacobians and abelian varieties of $\operatorname{GL}_2$-type.

\subsection{Modular Jacobians and abelian varieties of \texorpdfstring{$\operatorname{GL}_2$}{GL2}-type} \label{subsection:modular-jacobians}

We let $J_0(N)$ denote the Jacobian of the modular curve $X_0(N)$ and $J_0^{\mathrm{new}}(N)$ denote its newpart.
Recall that the Galois orbits of newforms $f \in S_2^{\mathrm{new}}(N)$ correspond to abelian varieties $A_f$ of $\operatorname{GL}_2$-type via the Eichler--Shimura construction. This abelian variety $A_f$ has $p$-Weil polynomial
\[
P_{A,p}(x)=\prod_{\tau : K_f \hookrightarrow \mathbb{C}} \left(x^2 - \tau(a_p(f))x + p\right),
\]
where $a_p(f)$ denotes the $p$-th Fourier coefficient of $f$ and where $\tau$ ranges over all embeddings of the Hecke field $K_f := \mathbb{Q}(\{a_n(f)\}_{n \geq 1})$ into $\mathbb{C}$. By \cite[Page 1]{MurtySinhaDecomp}, we have up to $\mathbb{Q}$-isogeny the decomposition
$$J_0(N) \sim_{\mathbb{Q}} \prod_{M | N} \prod_{f} A_f^{\sigma_0(N/M)},$$
where $f$ runs over Galois orbits of newforms in $S_2^{\mathrm{new}}(M)$. Each $A_f$ is $\mathbb{Q}$-simple, with
\[
\operatorname{Conductor}(A_f) = M^{[K_f : \mathbb{Q}]}, \qquad \dim A_f = [K_f : \mathbb{Q}], \qquad \operatorname{End}^0(A_f) \cong K_f.
\]Furthermore, the Atkin--Lehner involutions act naturally on $X_0(N)$ and hence on $J_0(N)$ by functoriality. Hence, we have the decomposition
\[
J_0(N) \sim_{\mathbb{Q}} \prod_{\sigma} J_0^{\sigma}(N),
\]
where the product runs over all Atkin-Lehner sign pattern $\sigma$ for $N$ (see \S1.5 for the definition). Furthermore, for each sign pattern $\sigma$, we have that 
\begin{align}
    J_0^{\sigma}(N) \sim_{\mathbb{Q}} J_0^{\nw, \sigma}(N) \times J_0^{\text{old}, \sigma}(N) \qquad \text{where $J_0^{\nw, \sigma}(N) \sim_{\mathbb{Q}} \prod_{f} A_f,$}
\end{align}
with $f$ running over Galois orbits of newforms in $S_2^{\nw, \sigma}(N)$. Moreover, the multiset $\Lambda_N^{\nw, \sigma}$ of $p$-Frobenius eigenvalues of $J_0^{\nw, \sigma}(N)$, is precisely the multiset of the roots $\alpha_i,\overline \alpha_i$ of $x^2 - a_p(f_i) x + p$ for newforms $f_i \in S_2^{\nw, \sigma}(N)$. This identity $\Lambda_N =\{\alpha_i, \overline{\alpha}_i\}_{i=1}^{\dim S_k^{\nw,\sigma}(N)}$
yields the correspondence between conjugate pairs of $p$-Frobenius eigenvalues and Fourier coefficients
\begin{align}
    \label{eqn:correspondence_conpair-pFrob-eig_and_four-coeff}
    (\alpha_i, \overline \alpha_i) = (\sqrt p e^{i\theta}, \overline{\sqrt p e^{i\theta}}\, ) 
    \quad
    \xleftrightarrow{\hspace{8mm}} 
    \quad
    a_f(p) = 2 \sqrt p \cos \theta.
\end{align}

Now, we recently showed in \cite[Theorem 1.1]{KMRX} that the normalized Fourier coefficients \\$\lrcb{\frac{1}{\sqrt p} a_{f_i}(p)}_{i=1}^{\dim S_k^{\nw,\sigma}(N)}$ are equidistributed over $[-2,2]$ via the $p$-adic Plancherel measure $d\mu_p(x)$ (with effective error bounds). By the correspondence \eqref{eqn:correspondence_conpair-pFrob-eig_and_four-coeff}, this then means that the normalized $p$-Frobenius eigenvalues $\frac{1}{\sqrt p} \Lambda_N$ are equidistributed over the unit circle $\{e^{i\theta} : \theta \in [0,2\pi]/(0 \sim 2\pi)\}$.
We write down the precise error bounds for this equidistribution in the following proposition.
\begin{proposition} \label{prop:p-Frob}
Let $d\hat\mu_p(\theta)$ denote the measure
\begin{align}
    d\hat\mu_p(\theta) := \frac{p + 1}{\pi}  \frac{\sin^2 \theta}{\left( p^{1/2} + p^{-1/2} \right)^2 - 4 \cos^2 \theta} \, d \theta \qquad \text{over } \theta \in [0,2\pi]/(0 \sim 2\pi).
\end{align}
Then there exists a universally bounded family of positive real-valued functions $\left\{C_p^{\nw, {\operatorname{AL}}}\right\}_{p}$ such that for coprime $p$ and $N$ and any Riemann-integrable test function $g$ over $[0,2\pi]/(0 \sim 2\pi)$,
\begin{align}
    \lrabs{
        \frac{1}{\# \Lambda_N^{\nw, \sigma}} \sum_{ \sqrt{p} \, e^{i \theta} \in \Lambda_N^{\nw, \sigma}} g\left( \theta \right) - \int_0^{2\pi} g(\theta) d \hat \mu_p(\theta)   
    } \le
    C_p^{\nw, {\operatorname{AL}}}(\log N) \cdot  \frac{\frac{1}{2} \delta(g) \log p}{\log N},
\end{align}
whenever $\Lambda_N^{\nw, \sigma}$ is non-empty. Here, $\delta(g)$ denotes the total variation of $g.$
\end{proposition}

\begin{remark}
    Rather than following immediately from \cite[Theorem 1.1]{KMRX}, Proposition \ref{prop:p-Frob} requires us to make a few modifications to the statement of \cite[Theorem 1.1]{KMRX} in particular, we must introduce the test function $g$ and the family of functions $\{C_p^{\nw, \operatorname{AL}}\}_p.$ We discuss these modifications in \S\ref{subsection:equidistribution-survey}, and we later derive an explicit form for $C_p^{\nw,\operatorname{AL}}$ in Appendix \ref{section:explicit-form-of-function}.
\end{remark}

\begin{remark}
    Analogous results to Proposition \ref{prop:p-Frob} also old for $J_0^{\sigma}(N), J_0(N),$ and $J_0^{\nw}(N)$, with corresponding families of functions $\left\{C_p^{\operatorname{AL}}\right\}_p,$ $\{C_p\}_p,$ and $\{C_p^\nw\}_p.$ These are all discussed in \S\ref{subsection:equidistribution-survey}. and the families of functions are derived in Appendix \ref{section:explicit-form-of-function}.
\end{remark}

\subsection{Vertical results}\label{subsection:vertical-results}

Our primary goal from the vertical perspective is to prove the titular claim.
\begin{theorem}\label{thm:finiteness-of-supersingularity}
    Fix a prime $p$. Then $J_0^{\nw, \sigma}(N)$ is supersingular modulo $p$ for only finitely many $N.$
    Consequently, the same result also holds for $J_0(N)$, $J_0^\nw(N),$ and $J_0^{\sigma}(N)$.
\end{theorem}
We prove Theorem \ref{thm:finiteness-of-supersingularity} via the following strategy. 
We have by Proposition \ref{prop:p-Frob} that the normalized $p$-Frobenius eigenvalues $\frac{1}{\sqrt p} \Lambda_N^{\nw, \sigma}$ are equidistributed over the unit circle via the measure $d\hat{\mu}_p(\theta)$. On the other hand, Lemma \ref{lem:cyclotomic_obstruction} shows that (roughly speaking) roots of unity cannot follow the distribution $d\hat{\mu}_p(\theta)$, so there can only exist finitely many $N$ such that $\frac{1}{\sqrt{p}} \Lambda_N^{\nw, \sigma}$ is composed of roots of unity. In particular, by the Manin--Oort criterion, this means that $J_0^{\mathrm{new},\sigma}(N)$ can only be supersingular modulo $p$ for finitely many $N$, as desired.

Now, because of Theorem \ref{thm:finiteness-of-supersingularity}, we can define following (finite) functions.

\begin{definition}\label{defn:M(p)}
    For any prime $p,$ we let
\begin{align}
    M(p) &:=  \max \left\{ N  : p \in \operatorname{SS}\!\left(J_0(N) \right)  \right\}\\
    M^{\nw}(p) &:= \max \left\{ N  : p \in \operatorname{SS}\!\left(J_0^\nw(N) \right)  \right\}\\
    M^{\operatorname{AL}}(p) &:= \max \left\{ N  : p \in \operatorname{SS}\!\left(J_0^\sigma(N) \right) \text{ for some sign pattern $\sigma$ for $N$} \right\}\\
    M^{\nw, \operatorname{AL}}(p) &:= \max \left\{ N  : p \in \operatorname{SS}\!\left(J_0^{\nw, \sigma}(N) \right) \text{ for some sign pattern $\sigma$ for $N$} \right\},
\end{align}
with the convention that $\max \varnothing = 0.$
\end{definition}

We then show the following effective version of Theorem \ref{thm:finiteness-of-supersingularity}.

\begin{proposition}\label{prop:intro-vertical}
    The function $M^{\nw, \operatorname{AL}}(p)$ is polynomially bounded in $p$. In particular, $\log_p M^{\nw, \operatorname{AL}}(p)$ is bounded by an (explicit) constant. The same result also holds for $M(p),$ $M^\nw(p),$ and $M^{\operatorname{AL}}(p)$.
\end{proposition}
We will also write down the bounding constants from this proposition. Explicit bounds for small $p$ will be given in Proposition \ref{prop:upper-bounds-on-E(p)}, explicit bounds for large $p$ will be given in Corollary \ref{cor:M(p)-for-sufficiently-large-p}, and asymptotic bounds will be given in Corollary \ref{cor:vertical-suprsingularity-with explicit exponent}.

We also remark here that an analog of Theorem \ref{thm:finiteness-of-supersingularity} holds for $J_1(N)$, $J(N),$ and more generally $J(X_\Gamma)$ for any congruence subgroup $\Gamma \subseteq \Gamma_0(N)$. We will furthermore prove in Corollary \ref{cor:shimura}  that for $DN > M^{\nw}(p),$ the Shimura Jacobian $J_0^D(N)$ can not be supersingular modulo $p$.

\subsection{Horizontal results} \label{subsection:horizontal-results}

We now address the horizontal perspective, where the newform $f$ (or level $N$) is fixed and $p$ varies. We set the following notation.

\begin{definition}
    For newforms $f \in S_2^{\nw}(N)$, let
    \begin{align}
        V(f) &:= \left\{ \text{$p$ prime} : a_p(f) = 0  \right\},\\
        V(N) &:= \left\{ \text{$p$ prime} : a_p(f) = 0 \quad \text{for all}\ \, f \in S_2^\nw(N) \right\}.
    \end{align}
\end{definition}

From the horizontal perspective, the relevant obstruction to supersingularity we exploit is no longer equidistribution but a ramification argument. If $p \in \operatorname{SS}(A_f) \setminus V(f),$ then $\alpha_p/\!\sqrt{p}$ is a root of unity of some bounded order $n$. Consequently, $p$ must ramify in $K_f(\zeta_n),$ which is only possible for finitely many primes. This leads to the following result.

\begin{proposition}\label{prop:structure-result}
    For any newform $f \in S_2^{\nw}(N),$
    \begin{align}
        \operatorname{SS}(A_f) \subseteq \left\{ p : p \mid \operatorname{disc}(K_f) \text{ or } p - 1 \mid 4[K_f : \mathbb{Q} ] \right\} \cup V(f).
    \end{align}
    In particular, $\operatorname{SS}(A_f) \setminus V(f)$ is a finite set.
\end{proposition}

Proposition \ref{prop:structure-result} immediately implies the following theorem, an analog of which also holds when $A_f$ is replaced by $J_0^\nw(N)$ and $J_0(N)$ (see Corollaries \ref{cor:horizontal-result-for-Jac} and \ref{cor:horizontal-result-for-Jac-full}).

\begin{theorem}\label{main-horizontal}
    Fix an abelian variety $A_f/\mathbb{Q}$ of $\operatorname{GL}_2$-type. Then for all \\ $p > C_f := \max \left\{ \operatorname{disc}(K_f), 4[K_f : \mathbb{Q}] + 1\right\},$ $A_f$ is supersingular modulo $p$ if and only if it is isogenous to a power of the unique (up to isogeny) supersingular elliptic curve over $\mathbb{F}_p$.
\end{theorem}

The next theorem follows from Proposition \ref{prop:structure-result} via an upper bound on $\operatorname{disc}(K_f)$ (to be proved in \S\ref{subsection:discriminant-bound}) and an upper bound on $V(f)$ (obtained from an effective Sato-Tate result by Thorner \cite{Thorner2021Effective}).

\begin{theorem}\label{thm:Hui-improvement}
    There exist universal constants $c_1$ and $c_2$ such that for any $X \geq 3$,
    \begin{enumerate}
        \item For each non-CM newform $f \in S_2^\nw(N),$
    \begin{align}
        \#  \operatorname{SS}(A_f) \cap \{p : p \leq X\} \leq c_1 \cdot \frac{\log (2N \log X)}{\sqrt{\log X}} \cdot \pi(X) + \frac{N}{3}.
    \end{align}
    \item For each non-CM newform $f \in S_2^\nw(N),$ if $L(s, \operatorname{Sym}^m (f))$ satisfies GRH for all $m \geq 0,$
    \begin{align}
        \#  \operatorname{SS}(A_f) \cap \{p : p \leq X\} \leq c_2 \cdot \frac{\log (2N \log X)}{X^{1/4}} \cdot \pi(X) + \frac{N}{3}.
    \end{align}
    \end{enumerate}
    Here, $\pi(X)$ denotes the prime counting function.
\end{theorem}

\begin{remark}
    Theorem \ref{thm:Hui-improvement} recovers Hui's result \cite[Theorem 1.1]{HuiSS2025} that $\operatorname{SS}(A)$ has density zero, in the special case where $A$ is an abelian variety of $\operatorname{GL}_2$-type. Moreover, in this case, our result improves Hui \cite[Theorem 2]{HuiSS2025}, which gives
    $$\# \operatorname{SS}(A) \cap \{ p : p \leq X\} = O\left( \frac{\pi(X)}{\left(\log X\right)^{1/2 - \epsilon}}  \right).$$
    Specifically, we replace the $(\log X)^\epsilon$ term with $\log \log X$ and make the dependence of the implied constant on the choice of $A_f$ explicit.

\end{remark}

In Corollary \ref{cor:Hui-improvement-for-Jacobian}, we will also prove an analog of Theorem \ref{thm:Hui-improvement} for each Jacobian $J_0^\nw(N)$, aside from the seven exceptional cases $N \in \{ 27, 32, 36, 49, 64, 108, 256\}$.
In these seven remaining cases, the Jacobian $J_0^\nw(N)$ turn out to actually have a positive proportion supersingular primes; see Proposition \ref{prop:seven-exceptional-levels}.

We also show that assuming the generalized Maeda conjecture (see Conjecture \ref{conj:gen-Maeda}), the upper bound in Theorem \ref{thm:Hui-improvement} can be improved to the absolute bound $\# \operatorname{SS}(A_f) \leq \frac{N}{3}$ for all newforms of 100\% of levels $N$, and likewise for 100\% of Jacobians $J_0^\nw(N)$ (Corollary \ref{cor:finiteness-assuming-Maeda}).

\subsection{Directions for further work}\label{subsection:further-work}

We conclude by discussing several directions for future work. In Theorem \ref{main-horizontal}, we proved that, for all primes $p > C_f$, $A_f$ is supersingular modulo $p$ if and only if $a_p(f)=0$. Thus, it would be interesting to determine whether this set of primes is finite or infinite: currently, Lang-Trotter heuristics predict that this set of primes is infinite precisely when $f$ has complex multiplication or the stable trace field $K_f^{\Gamma}$ has degree $1$ or $2$. We also restricted our study to abelian varieties of $\mathrm{GL}_2$-type over $\mathbb{Q}$: it would be natural to extend these questions to abelian varieties of $\mathrm{GL}_2$-type over number fields, as well as to abelian varieties associated to Hilbert modular forms, and to study supersingularity in these more general settings. Moreover, in our case, all such primes $p$ are not only supersingular, but also superspecial, by \cite[Lemma 2.2]{yu2011superspecial}: it would be interesting to study superspecial primes in more general settings as well.

In \cite{wangzhang2025}, it is proved that the set of ordinary primes for an abelian variety $A_f$ associated to a non-CM newform has positive density. A natural next step would be to make this result effective by obtaining explicit lower bounds or asymptotics for the counting function of ordinary primes. More generally, it is conjectured that every non-CM abelian variety $A/\overline{K}$ has ordinary reduction at a set of primes of density $1$: this is known for elliptic curves ($g=1$) by \cite{DMOS82} and for abelian surfaces ($g=2$) by \cite{Sawin2016}. Finally, one may ask more refined distributional questions: for example, given an abelian variety $A_f$ of dimension $g$, what is the distribution of primes whose reduction has a prescribed admissible $g$-dimensional Newton polygon?

\subsection{Some notation}

Lastly, we review some notation and functions that we will use throughout this paper.

\begin{enumerate}
    \item We let $Q \| N$ denote the condition that $Q$ exactly divides $N$; i.e., that $(Q, N/Q) = 1.$
    \item We use $\sigma$ to denote an Atkin-Lehner sign pattern for $N$; i.e., a multiplicative function
    $$\sigma : \{ Q : Q \| N \} \longrightarrow \{ \pm 1 \}.$$
    \item We let $S_k^{\sigma}(N) \subseteq S_k(N)$ denote the subspace of modular forms of fixed Atkin-Lehner sign-pattern $\sigma$ (see \cite[\S1.2]{KMRX}).
    \item We let $\omega(n)$ denote the number of prime divisors of $n.$ Note that Robin \cite[Th\'eor\`eme 11]{Rob83} proves that
    \begin{align}
        \omega(n) \leq 1.38407 \, \frac{\log n}{\log \log n}. \label{eqn:robin-omega-bound}
    \end{align}
    \item We let $\sigma_0(n)$ denote the number of divisors of $n.$ Note that Nicolas and Robin \cite[Th\'eor\`eme 1]{NR83} prove that
    \begin{align}
        \sigma_0(n) \leq 2^{1.53794 \, \frac{\log n}{\log \log n}} \label{eqn:nicolas-robin-sigma-bound}
    \end{align}
    \item We let $\psi(n) = n \prod_{p \mid n} \left(1 + \frac{1}{p} \right).$ Note that Sol\'e and Planat \cite[Corollary 2]{sole2011extreme} show
    \begin{align}
        \psi(n) &\leq n \cdot e^\gamma \log   \log n \label{eqn:Sole-Planat-psi-bound}
    \end{align}
    for $n \geq 31,$ where $\gamma$ is the Euler-Mascheroni constant.
    \item We let $\psi^\nw(n) = n \prod_{p^r \| n} \left( 1 - \frac{1}{p} - \frac{\mathds{1}_{r \geq 2}}{p^2} + \frac{\mathds{1}_{r \geq 3}}{p^3} \right).$ Note that the proof of \cite[Lemma 3.6]{KMRX} gives that
    \begin{align}
        \psi^{\nw}(n) \geq \frac{n \cdot C_{\text{Artin}}}{e^{\gamma} \log \log n - \frac{2.5}{\log \log n}} \label{eqn:bound-on-psi-new}
    \end{align}
    for $n \geq 223092870,$ where $C_{\text{Artin}} := \prod_p \lrp{1 - \frac{1}{p(p-1)}}$ denotes Artin's constant.
\end{enumerate}

\section{Vertical supersingularity}\label{section:vertical}

\subsection{Equidistribution of Hecke eigenvalues}\label{subsection:equidistribution-survey}

In this subsection, we recall several equidistribution results from which Proposition \ref{prop:p-Frob} and its analogs follow.

\begin{definition}
    The $p$-adic Plancherel measure over $[-2, 2]$ is given by
    \begin{align}
        d\mu_p(x) := \frac{p + 1}{2\pi} \cdot \frac{\sqrt{4 - x^2}}{\left( p^{1/2} + p^{-1/2}  \right)^2 -  x^2} \,dx. \label{eqn:definition-of-Plancherel}
    \end{align}
\end{definition}

\begin{definition} \label{def:mu-p}
    We let
    \begin{align}
        \| \mu_p\| &:= \sup_{x \in [-2, 2]} \frac{p + 1}{2\pi} \cdot \frac{\sqrt{4 - x^2}}{\left( p^{1/2} + p^{-1/2}  \right)^2 -  x^2} = \begin{cases}
        \frac{3\sqrt{2}}{4\pi} & p = 2\\
        \frac{\sqrt{3}}{2\pi} & p = 3\\
        \frac{3\sqrt{5}}{8\pi} & p = 5\\
        \frac{p - 1}{p\pi} & p \geq 7.
    \end{cases}
    \end{align}
\end{definition}

Let $\mathbf{T}_p' := \frac{1}{\sqrt{p}} \mathbf{T}_p$ denote the normalized $p$-th Hecke operator over $S_2(N)$. By Deligne's bound, we have $\eigen_{S_2(N)}(\mathbf{T}_p') \subseteq [-2, 2].$ Several results, beginning with Serre \cite[Th\'eor\`eme 1]{ser97} have shown that for certain subspaces $S \subseteq S_2(N)$ the set $\eigen_{S}(\mathbf{T}_p')$ becomes equidistributed via $d\mu_p(x)$ as $N \rightarrow \infty.$ Most recently, we proved the following result in \cite[Theorem 1.1]{KMRX}.

\begin{theorem}[{\cite[Theorem 1.1]{KMRX}}]
There exists an effectively computable constant $C_0$ such that 
    \begin{align}
        \left| \sum_{\lambda \in \operatorname{eigen}_{S_k^{\sigma}(N)}(\mathbf{T}'_p)} \chi_I(\lambda) - \int_{-2}^2 \chi_I(x) \, d \mu_p(x)  \right| \le C_0  \frac{\log p}{\log kN} \label{eqn:KMRX-thm-1.1}
    \end{align}
    for all primes $p$, levels $N$ coprime to $p$, admissible sign patterns $\sigma$, weights $k$, and intervals $I \subseteq [-2,2]$. Moreover, the same result also holds for $S_k^{\nw,\sigma}(N)$ (with a different constant $C_0^\nw$).
\end{theorem}

We make the following three modifications to this theorem.
\begin{enumerate}
    \item We specialize to $k = 2.$ Thus, with a slight adjustment of the constant $C_0$, we can replace the denominator $\log kN$ with $\log N.$
    \item By \cite[Theorem 3.1]{KMRX}, the constant $C_0$ can be taken to be arbitrarily close to $2\| \mu_2 \|$ if one restricts to large enough $N.$ For fixed $p,$ one can generalize $2\| \mu_2 \|$ to $2\| \mu_p \|$.
    \item We replace the indicator function $\chi_I$ with an arbitrary Riemann-integrable test function $g.$ This requires that we multiply the error term by $\delta(g)$ (see \cite[Corollary 2.6]{KMRX}).
\end{enumerate}
This strenthens \cite[Theorem 1.1]{KMRX} (for $k = 2$) by replacing \eqref{eqn:KMRX-thm-1.1} with the bound
\begin{align}
    \left|
\sum_{\lambda \in \operatorname{eigen}_{S_k^{\sigma}(N)}(\mathbf{T}'_p)} g(\lambda)
-
\int_{-2}^2 g(x)\, d\mu_p(x)
\right|
\leq C_p^{\operatorname{AL}}(\log N) \cdot \frac{\delta(g)\log p}{\log N} \label{eqn:Atkin-Lehner-equidistribution}
\end{align}
for a universally bounded family of positive real-valued functions $\left\{ C_p^{\operatorname{AL}}(N) \right\}_p$ with $\lim_{X \rightarrow \infty} C_p^{\operatorname{AL}}(X) = 2 \| \mu_p \|.$ An explicit formula for $C_p^{\operatorname{AL}}(X)$ is given in Appendix \ref{subsection:explicit-form-sigma}.

\begin{remark}
    The analog of \eqref{eqn:Atkin-Lehner-equidistribution} over the Atkin-Lehner newspace $S_2^{\nw, \sigma}(N)$ also holds, with the functions $\left\{ C_p^{\operatorname{AL}}(N) \right\}_p$ replaced by a different universally bounded family $\left\{ C_p^{\nw, {\operatorname{AL}}}(N) \right\}_p$ also satisfying $\lim_{X \rightarrow \infty} C_p^{\nw, {\operatorname{AL}}}(X) = 2 \| \mu_p \|.$ This is precisely the family of functions appearing in Proposition \ref{prop:p-Frob}, which follows immediately from a change of variables.
\end{remark}

\begin{remark}
    We also define a universally bounded family of functions $\left\{C_p^{\operatorname{AL}, \operatorname{pm}}\right\}_p$ such that for all prime levels $P \neq p,$
    \begin{align}
        \left|
        \frac{\# \operatorname{eigen}_{S_2^\sigma(P)}(\mathbf{T}'_p) \cap I}
        {\dim S_2^\sigma(P)}
        -
        \int_{-2}^2 \chi_I(x)\, d\mu_p(x)
        \right|
        \leq C_p^{\operatorname{AL},\text{pm}}(\log P) \cdot \frac{\log p}{\log P}. \label{eqn:Atkin-Lehner-equidistribution-prime}
    \end{align}
These functions still satisfy $\lim_{X \rightarrow \infty } C_p^{\operatorname{AL}, \text{pm}}(X) = 2 \| \mu_p \|,$ but restricting to prime levels causes the functions to converge much faster than $\{C_p^{\operatorname{AL}}\}_p.$
\end{remark}

The same modifications can likewise be made to Murty and Sinha's result \cite[Theorem 2]{MS07}, which effectivizes the equidistribution of the eigenvalues of $\mathbf{T}_p$ over the whole space $S_2(N).$ We modify this result as follows.

\begin{theorem}[{Modification of \cite[Theorem 2]{MS07}}]
Let $p$ be a prime, $I\subseteq [-2,2]$ be an interval and $N$ coprime to $p$. Then there exists a universally bounded family of positive real-valued functions $\{C_p(X)\}_p$ such that
\[
\left|
\frac{\# \operatorname{eigen}_{S_2(N)}(\mathbf{T}'_p) \cap I}
{\dim S_2(N)}
-
\int_{-2}^2 \chi_I(x)\, d\mu_p(x)
\right|
\leq C_p(\log N) \cdot \frac{\log p}{\log N}
\]
and such that $\lim_{X \rightarrow \infty } C_p(X) = \frac{3}{2} \| \mu_p \|.$
\end{theorem}

That $C_p$ can be taken to tend towards $\frac{3}{2} \| \mu_p \|$ can be seen directly from the proof of \cite[Theorem 2]{MS07}, or otherwise from an explicit formula for $C_p(X)$ that we derive in Appendix \ref{subsection:explicit-form-C}.

Likewise, we define and later provide explicit expressions for the analogous functions
\begin{enumerate}
    \item $C_p^{\text{pm}}$ satisfying $\lim_{X \rightarrow \infty} C_p^{\text{p}} =  \| \mu_p \|$; explicit form derived in Appendix \ref{subsection:explicit-form-prime-divisors}
    \item $C_p^{\text{pp}}$ satisfying $\lim_{X \rightarrow \infty} C_p^{\text{pp}} =  \frac{3}{2}\| \mu_p \|$; explicit form derived in Appendix \ref{subsection:explicit-form-prime-divisors}
    \item $C_p^{\nw}$ satisfying $\lim_{X \rightarrow \infty} C_p^{\nw} = \frac{3}{2} \| \mu_p \|$; explicit form derived in Appendix \ref{subsection:explicit-form-newspace}
    \item $C_p^{\nw, \text{pp}}$ satisfying $\lim_{X \rightarrow \infty} C_p^{\nw, \text{pp}} =  \frac{3}{2}\| \mu_p \|$; explicit form derived in Appendix \ref{subsection:explicit-form-newspace}.
\end{enumerate}
Here, $C_p^{\text{pp}}$ is defined analogously to $C_p^{\text{pm}},$ except one ranges the level $N$ over all prime powers rather than just primes. $C_p^{\nw, \text{pp}}$ is defined likewise. All of these families of functions $\{C_p^{\text{condition}}\}_p$ can be used to derive analogs of Proposition \ref{prop:p-Frob}.

\subsection{Proof that only finitely many \texorpdfstring{$J_0(N)$}{J0(N)} are supersingular modulo \texorpdfstring{$p$}{p}}\label{subsection:proof-of-main-vertical-theorem}

We study when the newparts $J_0^{\nw,\sigma}(N)$ of the Atkin-Lehner isotypic components of the Jacobian $J_0(N)$ are supersingular modulo $p$. We begin by setting the following notation.

\begin{definition}
    For an algebraic number $\alpha \in \overline{\mathbb{Q}},$ we let
$\operatorname{Conj}_{\operatorname{Gal(\overline{\mathbb{Q}}/\mathbb{Q}})} ( \alpha)$
denote the set of $\operatorname{Gal(\overline{\mathbb{Q}}/\mathbb{Q}})$-conjugates of $\alpha.$
\end{definition}

We prove the following proposition, which immediately implies Theorem \ref{thm:finiteness-of-supersingularity} and Proposition \ref{prop:intro-vertical}, since the family $\{ C_p(X) \}_p$ and all its analogs are universally bounded (see \S\ref{subsection:equidistribution-survey}).

\begin{proposition}\label{prop:vertical-supersingularity}
Fix a prime \(p\). Then $J_0^{\nw,\sigma}(N)$ can only be supersingular modulo $p$ for
\begin{align}
    N \leq \begin{cases}
        p^{150 \,C_p^{\nw,\operatorname{AL}}(\log N)} & \text{if } p = 2\\
        p^{32\frac{p}{p - 2} \, C_p^{\nw,\operatorname{AL}}(\log N)} & 
        \text{if } p \geq 3.
    \end{cases}
\end{align}
The analogous statement also holds for $J_0(N),$ $J_0^\nw(N),$ and $J_0^{\sigma}(N),$ when $C_p^{\nw, {\operatorname{AL}}}$ is replaced by $C_p,$ $C_p^\nw,$ and $C_p^{\operatorname{AL}},$ respectively.
\end{proposition}
\begin{proof}
We focus on the case of $J_0^{\nw, \sigma}(N)$, and the proof of all the other analogs is identical.

Set $d_N:=\dim S_2^{\mathrm{new},\sigma}(N)=\dim J_0^{\nw, \sigma}(N)$, and let $f_1,\dots,f_{d_N}$
denote the newforms in \(S_2^{\mathrm{new},\sigma}(N)\).
Let $\Lambda_N^{\nw, \sigma}$ denote the multiset of \(p\)-Frobenius eigenvalues for $J_0^{\nw, \sigma}(N)$,
\begin{align}
    \Lambda_N^{\nw, \sigma} := \lrcb{ \alpha_j, \overline \alpha_j : 1 \le j \le d_N}, \qquad \text{where } \alpha_j, \overline\alpha_j \text{ are the roots of $x^2-a_p(f_j)x+p$.}
\end{align}
The multisets
$
\frac{1}{\sqrt p}
\Lambda_N^{\nw, \sigma}
$
are equidistributed on the unit circle
\(
\{e^{i\theta}:0\le\theta<2\pi\}
\)
via the measure \(d\hat\mu_p(\theta)\) and with explicit error bound given by Proposition~\ref{prop:p-Frob}. 

Now, recall that by the Manin--Oort criterion, $J_0^{\nw, \sigma}(N)$ is supersingular if and only if each eigenvalue $\alpha \in \Lambda_N^{\nw, \sigma}$ is of the form $\alpha = \sqrt p\,\zeta$ for some root of unity \(\zeta\). Additionally, recall that by the Honda–Tate theorem, simple abelian varieties over $\mathbb{F}_p$ correspond to $\mathrm{Gal}(\overline{\mathbb{Q}}/\mathbb{Q})$-conjugacy classes of $p$-Weil numbers, via the bijection that sends an abelian variety to the set of its $p$-Frobenius eigenvalues. Consequently, arbitrary abelian varieties over $\mathbb{F}_p$ correspond to multiset unions of $\mathrm{Gal}(\overline{\mathbb{Q}}/\mathbb{Q})$-conjugacy classes of $p$-Weil numbers.
Hence if $J_0^{\nw, \sigma}(N)$ is supersingular, then
$\Lambda_N^{\nw, \sigma}$ can be written as a multiset union of $\text{Gal}(\overline{\mathbb{Q}}/\mathbb{Q})$-conjugacy classes of elements of the form $\sqrt p\,\zeta$; i.e.,
\begin{align}
    \Lambda_N^{\nw, \sigma} = \bigcup_{\zeta \in Z_N} \mathrm{Conj}_{\text{Gal}(\overline{\mathbb{Q}}/\mathbb{Q})}(\sqrt p \,\zeta),
\end{align}
where $Z_N$ is some multiset of roots of unity. 
By Lemma \ref{lem:conjugates-shape}, this then means that $\frac{1}{\sqrt{p}} \Lambda_N^{\nw, \sigma}$ can be written as a multiset union of $\Phi_n^+$-imitators (see Definition \ref{def:phi-n-plus-imitators}). But by Lemma \ref{lem:cyclotomic_obstruction}, this can only occur for 
$
    N \leq
    \begin{cases}
         2^{150 \, C_p^{\nw,\operatorname{AL}}(\log N)} 
        & \text{ if } p=2 \\
        p^{ 32\frac{p}{p-2} \, C_p^{\nw,\operatorname{AL}}(\log N)}
        & \text{ if } p \geq 3
    \end{cases}
$. This completes the proof.
\end{proof}

\begin{corollary}
For a fixed prime $p$ and $N$ sufficiently large, there exists a $\mathbb{Q}$-simple abelian variety $A_f$ with conductor $N^{\dim A_f}$ that is not supersingular modulo $p.$
\end{corollary}

\begin{remark}
The statement of Proposition \ref{prop:vertical-supersingularity} also holds for $J_1(N)$, $J(N)$ and $J(X_\Gamma):=\text{Jac}(X_\Gamma)$ for any congruence subgroup $\Gamma$ with $\Gamma \subset \Gamma_0 (N)$. This follows from the fact that if $\Gamma \subset \Gamma_0(N)$, 
then all $A_f$ for $f \in S_2(\Gamma_0(N))$ appear in the 
$\mathbb{Q}$-isotypic decomposition of $J(X_\Gamma)$, 
because the pullback of the 
forgetful map 
$\pi:X_\Gamma \to X_0(N)$ is a map $\pi^*:J_0(N) \to J(X_\Gamma)$ 
with finite kernel.
\end{remark}

We also obtain an analogous result for the Shimura Jacobians $J_0^D(N) := \text{Jac}(X_0^D(N))$. Let $B/\mathbb{Q}$ be an indefinite quaternion algebra of discriminant $D$, so $D$ is a squarefree product of an even number of distinct primes. Then the Shimura curve $X_0^D(N)$ is attached to an Eichler order of level $N$ in $B$. 

\begin{corollary}\label{cor:shimura}
Fix a prime $p$. Then $J_0^D(N)$ is not supersingular modulo $p$ for all coprime integers $D$ and $N$ with $DN > M^{\nw}(p)$.
\end{corollary}

\begin{proof}
By the Jacquet--Langlands correspondence, for all $f \in S_2^{\nw}(DN),$ we have that $A_f$ is a subvariety of $J_0^D(N),$ so $\prod_{f \in S_2^{\nw}(DN)/\sim} A_f \sim J_0^{\nw}(DN) \subset J_0^D(N)$. Therefore, whenever $DN > M^{\nw}(p)$, we have $p\notin \operatorname{SS}(J_0^{\nw}(DN))$, so $p\notin \operatorname{SS}(J_0^{D}(N))$.
\end{proof}

Finally, we establish Lemmas \ref{lem:conjugates-shape} and \ref{lem:cyclotomic_obstruction}, which were used in the proof of Proposition \ref{prop:vertical-supersingularity}.

\begin{definition}
    \label{def:phi-n-plus-imitators}
    Let $\Phi_n$ denote the set of primitive $n$-th roots of unity, and define $\Phi_n^+ := \{\zeta \in \Phi_n : \Im(\zeta) \ge 0\}$.
    Then we say that a set $\Psi_n \subseteq S^1$ is a \textit{$\Phi_n^+$-imitator} if there exists a bijection
    \begin{align}
        \xi :\Phi_n^+ \longleftrightarrow \Psi_n
    \end{align}
    such that $\xi(\zeta) = \pm \zeta^{\pm 1}$ for each $\zeta \in \Phi_n^+$. 
    Equivalently, $\Psi_n$ is a $\Phi_n^+$-imitator if
    \begin{align}
        \big\{|x| + i|y| : x+iy \in \Psi_n \big\} 
        = 
        \big\{|x| + iy : x+iy \in \Phi_n^+\big\} 
        \qquad \text{(as multisets)}.
    \end{align}
    For example, $\{-1\}$ is both a $\Phi_1^+$-imitator and a $\Phi_2^+$-imitator, and $\lrcb{e^{-2\pi/7}, e^{4\pi/7}, -e^{-6\pi/7}}$ is a $\Phi_7^+$-imitator.
\end{definition}

In Lemma \ref{lem:conjugates-shape}, we compute the form of $\operatorname{Conj}_{\operatorname{Gal}(\overline{\QQ}/\QQ)}(\sqrt p\,\zeta)$.
\begin{lemma}\label{lem:conjugates-shape}
For every prime $p$ and every
primitive $n$-th root of unity $\zeta$, 
$\frac{1}{\sqrt p} \operatorname{Conj}_{\operatorname{Gal}(\overline{\QQ}/\QQ)}(\sqrt p\,\zeta)$ is equal to a finite disjoint union of $\Phi_n^+$-imitators.
\end{lemma}

\begin{proof}
Let $L=\QQ(\sqrt p,\zeta)$. Since every Galois conjugate of $\sqrt p\,\zeta$ lies in $L$, first observe that
\begin{align}
    \operatorname{Conj}_{\operatorname{Gal}(\overline{\QQ}/\QQ)}(\sqrt p\,\zeta)=\big\{\tau(\sqrt p\,\zeta):\tau\in\operatorname{Gal}(L/\QQ)\big\} \qquad \text{(at least, as sets)}.
\end{align} 
Then we break into cases.

{\bf Case 1: $\sqrt p \notin \QQ(\zeta)$}. \\ 
Since $\sqrt p\notin\QQ(\zeta)$, we have that $\QQ(\sqrt p)\cap\QQ(\zeta)=\QQ$, so 
\[
\operatorname{Gal}(L/\QQ)\cong\operatorname{Gal}(\QQ(\sqrt p)/\QQ)\times\operatorname{Gal}(\QQ(\zeta)/\QQ).
\] 
This means that the automorphisms in $\mathrm{Gal}(L/\QQ)$ are precisely those of the form $(\rho,\tau_a)$, where
$\rho(\sqrt p)=\pm\sqrt p$ and 
$\tau_a(\zeta)=\zeta^a$ for $a\in(\ZZ/n\ZZ)^\times$. 
Hence
\begin{align}
    \frac{1}{\sqrt p} \operatorname{Conj}_{\operatorname{Gal}(\overline{\QQ}/\QQ)}(\sqrt p\,\zeta)
    &= \frac{1}{\sqrt p} \Big\{\tau(\sqrt p\,\zeta):\tau\in\operatorname{Gal}(L/\QQ)\Big\} \\
    &= \Big\{ \pm \zeta^a : a\in(\ZZ/n\ZZ)^\times \Big\} \\
    &= \Phi_n \cup -\Phi_n \\
    &= 
    \begin{cases}
        \Phi_n \sqcup -\Phi_n & \text{if } 4 \nmid n \\
        \Phi_n & \text{if } 4 \mid n.
    \end{cases}
\end{align}

This implies the desired result since $\Phi_n$ is a finite disjoint union of $\Phi_n^+$-imitators; namely
\begin{align}
    \Phi_n = 
    \begin{cases}
        \{ z : z \in \Phi_n^+\} 
        & \text{if } n=1,2 \\
        \{z : z \in\Phi_n^+\}\sqcup\{z^{-1} : z \in \Phi_n^+\}
        & \text{if } n \ge 3.
    \end{cases} 
\end{align}

\textbf{Case 2: $\sqrt p \in \QQ(\zeta)$.}\\ 
In this case, we have
\[
\operatorname{Gal}(L/\QQ)
=
\operatorname{Gal}(\QQ(\zeta)/\QQ),
\] 
and so the automorphisms in $\mathrm{Gal}(L/\QQ)$ are precisely those of the form $\tau_a$, where
$\tau_a(\zeta)=\zeta^a$ for $a\in(\ZZ/n\ZZ)^\times$. Now, each of these automorphisms must send $\sqrt p$ to $\pm\sqrt p$. This then implies that 
$\tau_a(\sqrt p)=\chi(a)\sqrt p$, where $\chi$ is some quadratic Dirichlet character modulo $n$.

Consequently,
\begin{align}
    &\frac{1}{\sqrt p} \operatorname{Conj}_{\operatorname{Gal}(\overline{\QQ}/\QQ)}(\sqrt p\,\zeta) \\
    &=
    \frac{1}{\sqrt p} \Big\{\tau(\sqrt p\,\zeta):\tau\in\operatorname{Gal}(L/\QQ)\Big\} \\
    &=
    \frac{1}{\sqrt p} \Big\{\tau_a(\sqrt p) \tau_a(\zeta) : a \in (\ZZ/n\ZZ)^{\times}\Big\} \\
    &=
    \Big\{\chi(a)\,\zeta^a : a \in (\ZZ/n\ZZ)^{\times}\Big\} \\
    &=
    \Big\{\chi(a)\,\zeta^a, \chi(n-a) \zeta^{n-a} : a \in (\ZZ/n\ZZ)^{\times},\ 0 < a < \tfrac{n}{2} \Big\} \\
    &=
    \Big\{\chi(a)\,\zeta^a : a \in (\ZZ/n\ZZ)^{\times},\ 0 < a < \tfrac{n}{2} \Big\} 
    \cup 
    \Big\{\chi(-a) \zeta^{-a} : a \in (\ZZ/n\ZZ)^{\times},\ 0 < a < \tfrac{n}{2} \Big\}
    \\ 
    &= 
    \begin{cases}
        \Big\{\chi(a)\,\zeta^a : a \in (\ZZ/n\ZZ)^{\times},\ 0 < a < \tfrac{n}{2} \Big\}
        &\hspace{-45mm}
        \quad\text{if } 4\mid n,\ \chi(1-\tfrac{n}{2})=-1,  \\
        \Big\{\chi(a)\,\zeta^a : a \in (\ZZ/n\ZZ)^{\times},\ 0 < a < \tfrac{n}{2} \Big\} 
        \sqcup
        \Big\{\chi(-a)\,\zeta^{-a} : a \in (\ZZ/n\ZZ)^{\times},\ 0 < a < \tfrac{n}{2} \Big\}
        \label{eqn:temp-conj-formula-disjoint-union}
        \\
        &\hspace{-17.5mm} \text{otherwise.}  
    \end{cases}  
\end{align}
Here, the last step follows from the equivalence
\begin{align} \label{eqn:temp-twisted-roots-of-unity-overlap-equiv-condition}
    &\quad\ \ \,  \chi(a) \zeta^a = \chi(-a') \zeta^{-a'}  \\
    &\text{iff}\quad
    \chi(-aa') =\zeta^{a+a'}  
    & \text{(since $\chi$ is quadratic)} \\
    &\text{iff}\quad
    \chi(-aa') = -1
    \text{ and }
    a+a' = \tfrac{n}{2} 
    &\text{(since $\chi$ is quadratic and $0 < a,a' < \tfrac{n}{2}$)}
    \\
    &\text{iff}\quad
    4 \mid n,\ \chi\lrp{-a(\tfrac{n}{2} - a)}=-1, \text{ and } 
    a' = \tfrac{n}{2}-a 
    &\text{(since $\tfrac{n}{2} - a \in (\ZZ/n\ZZ)^\times$ only when $4 \mid n$)} 
    \\
    &\text{iff}\quad
    4 \mid n,\ \chi\lrp{1-\tfrac{n}{2}}=-1, \text{ and } 
    a' = \tfrac{n}{2}-a. \\
    & &
    \hspace{-80mm}
    \Big(\text{since $\chi(-a(\tfrac{n}{2}-a)) = \chi(a^2 - a^2 \tfrac{n}{2} + a^2 \tfrac{n}{2} - a \tfrac{n}{2}) = \chi(a^2 (1-\tfrac{n}{2}) + \tfrac{a(a-1)}{2}n) = \chi(1 - \tfrac{n}{2})$}\Big). 
\end{align}
Also, note that all the elements are distinct within each set of \eqref{eqn:temp-conj-formula-disjoint-union}, since $\zeta^{a-a'} \ne \pm 1$ for $a \ne a'$. 

Finally, observe that $\big\{\zeta^a : a \in (\ZZ/n\ZZ)^{\times},\ 0 < a < \tfrac{n}{2} \big\}$ is a $\Phi_n^+$-imitator since
\begin{align}
    \Big\{ \big(\zeta^a\big)^{\mathrm{sgn}\,\Im(\zeta^a)} :  a \in (\ZZ/n\ZZ)^{\times},\ 0 < a < \tfrac{n}{2} \Big\} = \Phi_n^+.\label{eqn:temp-set-equal-to-Phi_n+} 
\end{align}
Hence $\big\{\chi(a)\,\zeta^a : a \in (\ZZ/n\ZZ)^{\times},\ 0 < a < \tfrac{n}{2} \big\}$ 
and 
$\big\{\chi(-a)\,\zeta^{-a} : a \in (\ZZ/n\ZZ)^{\times},\ 0 < a < \tfrac{n}{2} \big\}$ 
are also $\Phi_n^+$-imitators, so \eqref{eqn:temp-conj-formula-disjoint-union}
yields the desired result.
\end{proof}

Finally, in Lemma \ref{lem:cyclotomic_obstruction}, we show that $\Phi_n^+$-imitators cannot be equidistributed via $d\hat \mu_p(\theta)$.
\begin{lemma}\label{lem:cyclotomic_obstruction}
    Fix a prime $p$ and let $\{L_N\}_{N\ge 1}$ be a sequence of multisets contained in the unit circle $\{e^{i\theta} : \theta \in [0,2\pi]/(0\sim2\pi)\}$ that is $d\hat \mu_p(\theta)$-equidistributed with error bound
    \begin{equation}
        \left| \frac{1}{\# L_N} \sum_{e^{i\theta} \in L_N} g( \theta) \ -\  \int_{0}^{2\pi} g(\theta) \,d\hat{\mu}_p(\theta) \right| \le C(\log N) \cdot  \frac{\frac{1}{2} \delta(g)\log p}{\log N}. 
        \label{eqn:main-thm-for-general-function-g-restated-forr-supersingularity}
    \end{equation}
    for some positive real-valued function $C$ (c.f. Proposition \ref{prop:p-Frob}). Then $L_N$ cannot be written as a multiset union of $\Phi_n^+$-imitators for 
    $$
    N >
    \begin{cases}
         2^{150 \, C(\log N)} 
        & \text{ if } p=2, \\
        p^{ 32\frac{p}{p-2} \, C(\log N)}
        & \text{ if } p \geq 3.
    \end{cases}
    $$.
\end{lemma}
\begin{proof}
    We first prove the desired result for the case of $p \ge 3$. 
    Let 
    \begin{align}
        B_0 := 
        \lrcb{
            m,2m : m=3
        } 
        \sqcup\, 
        \lrcb{
            4
        }
        = \lrcb{3,4,6},
    \end{align}
    and define the function
    \begin{align} \label{eqn:temp-g0-def}
        g_0(x) := -2 \cos(2\theta) \chi_{I_0}(\theta), 
        \qquad \text{where} \qquad
        I_0 := ([0,2\pi]/(0\sim2\pi)) \setminus\bigcup_{n \in B_0} \big\{\arg(\zeta) : \zeta \in \Phi_n\big\}.
    \end{align}
    Note that $B_0$ is constructed in this way so that for any primitive $n$-th root of unity $\zeta \in \Phi_n$,
    \begin{align} 
        \arg \pm \zeta^{\pm 1} \in I_0 
        \qquad &\text{iff} \qquad  
        \arg \zeta \in I_0 
        \qquad \text{iff} \qquad
        n \notin B_0, \\
        \intertext{and so}
        g_0\lrp{\arg \pm \zeta^{\pm1} }
        &= g_0\lrp{\arg \zeta }
        = -2 \cos\!\big(2 \arg \zeta\big) \mathds{1}_{n \notin B_0}  
        \label{eqn:temp-formula-for-g0(2Re(+-zeta^+1))}.
    \end{align}
    
    Now, the total variation of $g_0(\theta)$ over $[0,2\pi]/(0\sim2\pi)$ is given by
    \begin{align}
        \delta(g_0) &= \delta\lrp{\big[\theta \mapsto -2 \cos(2\theta)\big]} + \sum_{\theta \notin I_0} 2 \cdot \lrabs{-2 \cos(2\theta)} = 32.
    \end{align}
    Hence by the $d\hat\mu_p(\theta)$-equidistribution of $\{L_N\}_{N\ge 1}$, we have that 
    \begin{align} \label{eqn:temp_integral_2-x2_}
        \left| \frac{1}{\# L_N} \sum_{e^{i\theta} \in L_N} g_0\!\left( \theta  \right) \ -\  \int_{0}^{2\pi} g_0(\theta) \,d\hat\mu_p(\theta) \right| \le C(\log N) \frac{16\log p}{\log N}.
    \end{align}
    
    Now, the integral in \eqref{eqn:temp_integral_2-x2_} is given by
    \begin{align} 
        \int_{0}^{2\pi} g_0(\theta) d\hat \mu_p(\theta) 
        &= \int_{0}^{2\pi} -2 \cos(2 \theta)\cdot \frac{p + 1}{\pi}  \frac{\sin^2 \theta}{\left( p^{1/2} + p^{-1/2} \right)^2 - 4 \cos^2 \theta} \, d \theta \\
        &= 4\int_{0}^{\pi/2} -2 \cos(2\theta) \cdot \frac{p + 1}{\pi}  \frac{\sin^2 \theta}{\left( p^{1/2} + p^{-1/2} \right)^2 - 4 \cos^2 \theta} \, d \theta \\
        &= \frac{p+1}{\pi} \int_{0}^{\pi/2}
        \frac{-8 \cos(2 \theta)\sin^2 \theta}{c-4\cos^2\theta} d\theta \qquad \text{where $c = \left( p^{1/2} + p^{-1/2} \right)^2$} \\
        &= 
        \frac{p+1}{\pi}
        \lrb{
            (c - 2)
            \sqrt{\tfrac{c - 4}{c}}  
            \arctan\left(\sqrt{\tfrac{c}{c-4}} \tan\left({\theta}\right)\right) - \sin(2\theta) - (c-4) \theta
        }_{\theta=0}^{\theta=\pi/2} \\
        &= \frac{p+1}{\pi} \lrb{
            (c-2) \sqrt{\tfrac{c-4}{c}} \frac{\pi}{2} - (c-4) \frac{\pi}{2} 
        } \\
        &= \frac{p+1}{2} \lrb{
            \lrp{p+p^{-1}} \sqrt{\frac{\left( p^{1/2} - p^{-1/2} \right)^2}{\left( p^{1/2} + p^{-1/2} \right)^2}}  - \left( p^{1/2} - p^{-1/2} \right)^2
        } \\
        &= \frac{p+1}{2} \lrb{\frac{2(p-1)}{p(p+1)}} \\
        &= 1 - \frac{1}{p}.
    \end{align}

    On the other hand, suppose that $L_N$ can be written as a multiset union of $\Phi_n^+$-imitators, say
    $
        L_N = \bigcup_{n \in D_N}  \Psi_n
    $. 
    Then the summation in \eqref{eqn:temp_integral_2-x2_} is equal to the weighted average
    \begin{align}
        \frac{1}{\# L_N} \sum_{\zeta \in L_N} g_0(\arg \zeta) 
        =
        \sum_{n \in D_N} \lrp{\frac{\# \Psi_n}{\# L_N}} \cdot 
        \frac{1}{\# \Psi_n}  \sum_{\zeta \in \Psi_n}  g_0\lrp{\arg\zeta}.
    \end{align}
    Moreover, each of these terms in this weighted average is bounded above by 
    \begin{align}
        \frac{1}{\# \Psi_n}  \sum_{\zeta \in \Psi_n}  g_0\lrp{\arg \zeta} 
        &=
        \frac{1}{\# \Phi_n^+}  \sum_{\zeta \in \Phi_n^+}  g_0\lrp{\arg \xi(\zeta)} 
        & \hspace{-25mm}\text{(where $\xi(\zeta) = \pm \zeta^{\pm 1}$ for each $\zeta \in \Phi_n^+$)}
        \\
        &=
        \frac{1}{\# \Phi_n^+} \sum_{\zeta \in \Phi_n^+} 
        -2 \cos(2 \arg \zeta)
        \mathds{1}_{n \notin B_0} 
        & \text{(by \eqref{eqn:temp-formula-for-g0(2Re(+-zeta^+1))})}
        \\
        &=
        \frac{-2 \cdot \mathds{1}_{n \notin B_0}}{\# \Phi_n^+}
        \sum_{e^{i\theta} \in \Phi_n^+} \cos(2\theta) 
        \\
        &= 
        \frac{-2 \cdot \mathds{1}_{n \notin B_0}}{\# \Phi_n} \,\sum_{e^{i\theta} \in \Phi_n} \cos(2\theta)
        \\
        &= \frac{-2 \cdot \mathds{1}_{n \notin B_0}}{\phi(n)} \mathrm{Re}\lrb{\sum_{\zeta \in \Phi_n} \zeta^2 }  \\
        &= \frac{-2 \cdot \mathds{1}_{n \notin B_0}}{\phi(n)}  \mathrm{Re}\lrb{
            \frac{\gcd(n,4)}{\gcd(n,2)} \sum_{\zeta \in \Phi_{n/\!\gcd(n,2)}} \zeta
        } \\
        &= \frac{-2 \cdot \mathds{1}_{n \notin B_0}}{\phi(n)} \frac{\gcd(n,4)}{\gcd(n,2)}
            \,\mu\lrp{\frac{n}{\gcd(n,2)}} \\
        &\le \frac{1}{2}, \label{eqn:temp-upper-bound-on-Psi_n-summation-expressions} 
    \end{align}
    implying that the summation in \eqref{eqn:temp_integral_2-x2_} is also bounded above by $\frac{1}{2}$. Note that the last step \eqref{eqn:temp-upper-bound-on-Psi_n-summation-expressions} follows from the fact that
    $\phi(n) \ge 0.407\, n^{7/8} \ge 8$ for $n \ge 31$ (by the technique of \cite{ross-explicit-bounds}), and from direct computation for $1 \le n \le 30$.

    Hence for all $L_N$ that can be written as a multiset union of $\Phi_n^+$-imitators, we have by \eqref{eqn:temp_integral_2-x2_} that $\big(1-\frac{1}{p}\big) - \frac{1}{2} = \frac{p-2}{2p} \le C(\log N)\frac{16 \log p}{\log N}$. This then implies that
    $
        N \le p^{32 \frac{p}{p-2} \, C(\log N)}
    $,
    verifying the desired result for $p \ge 3$.

    For the case of $p=2$, the only modification needed is that one should instead choose $B_0 := \lrcb{m,2m : m=3,5} \sqcup \lrcb{4} = \lrcb{3,4,5,6,10}$. For this choice of $B_0$, the total variation of $g_0$ will be given by $\delta(g_0) = 49.888... \le 50$, and the upper bound \eqref{eqn:temp-upper-bound-on-Psi_n-summation-expressions} will become $\frac{1}{3}$.
    Hence for all $L_N$ that can be written as a multiset union of $\Phi_n^+$-imitators, we have by an identical argument as for the case of $p \ge 3$ that $\big(1-\frac{1}{2}\big) - \frac{1}{3} = \frac{1}{6} \le C(\log N) \frac{25 \log 2}{\log(N)}$. This then implies that
    $
        N \le 2^{150 C(\log N)}
    $,
    verifying the desired result for $p = 2$.
\end{proof}

\subsection{Upper bounds for \texorpdfstring{$\log_p M(p)$}{log p M(p)}}\label{subsection:explicit-bounds-on-M}

In this subsection, we find explicit upper bounds for $\log_p M^{\nw, \operatorname{AL}}(p)$ and its analogs, as defined in Definition \ref{defn:M(p)}. This is accomplished via explicit bounds on $C_p(N)$ and its analogs, which will be computed in Appendix \ref{section:explicit-form-of-function}.

First, we have the following corollary of Proposition \ref{prop:vertical-supersingularity}.

\begin{corollary}\label{cor:E(p)-relation}
    Let $E$ be a real number such that
    \begin{align}
        X > C_p\left( X \cdot \log p  \right) \cdot 
        \begin{cases}
            150 & \text{if } p = 2\\
            32 \frac{p}{p - 2} 
            & \text{if }
            p \ge 3
        \end{cases} \label{eqn:E(p)-relation}
    \end{align}
    for all $X > E$. Then $\log_p M(p) \leq E.$
\end{corollary}

\begin{proof}
    Note that there must exist such a value of $E$ since $C_p(X \log p)$ is universally bounded. Note that Proposition \ref{prop:vertical-supersingularity} guarantees that
    $$\log_p M(p) \leq C_p \left( \log_p M(p) \cdot \log p \right) \cdot \begin{cases}
            150 & \text{if } p = 2\\
            32 \frac{p}{p - 2} 
            & \text{if }
            p \ge 3
        \end{cases}.$$
    Hence by assumption, $\log_p M(p)$ cannot be strictly greater than $E.$
\end{proof}

\begin{remark}\label{rmk:extension-of-corollary-formula-for-M}
For any prime $p,$ define
\begin{align}
    M^{\text{pm}}(p) &:= \max \{ \text{primes $P$} : p \in \operatorname{SS}\!\left( J_0(P) \right) \}  \\
    M^{\text{pp}}(p) &:=  \max \left\{ \text{prime powers $Q$}  : \text{$p \in \operatorname{SS}\!\left(J_0(Q) \right)$}  \right\}
\end{align}
Also define $M^{\nw, \text{pp}}(p)$ and $M^{\operatorname{AL}, \text{pm}}(p)$ analogously. Note it is not necessary to define $M^{\nw, \operatorname{pm}}(p)$ or $M^{\nw, \operatorname{AL}, \operatorname{pm}}$ since for prime level $P,$ $J_0^\nw(P) \sim J_0(P)$ and $J_0^{\nw, \sigma}(P) \sim J_0^\sigma(P)$.

    Then the statement of Corollary \ref{cor:E(p)-relation} is true for any variant $M^{\text{condition}}(p)$ of $M(p)$ when $C_p(N)$ is replaced by $C_p^{\text{condition}}(N).$ This is clear from the definitions.
\end{remark}

Using a combination of Corollary \ref{cor:E(p)-relation} and the explicit forms of $C_p$ and its analogs from Appendix \ref{section:explicit-form-of-function}, we can bound any particular value of $M(p)$ and its analogs.

\begin{proposition} \label{prop:upper-bounds-on-E(p)}
    The following table contains upper bounds for $\log_p M(p)$ and its analogs for the first few primes.

    \begin{center}
    \begin{tabular}{|c||c|c|c|c|c|c|c|c|c|c|}
    \hline
       $p$ & $2$ & $3$ & $5$ & $7$ & $11$ & $13$ & $17$ & $19$ \\
    \hline
    \hline
       $\log_p M(p)$ & $134.791$ & $74.2022$ & $42.9048$ & $36.2355$ & $32.0027$ & $30.9110$ & $29.4828$ & $28.9940$  \\
    \hline
        $\log_p M^{\operatorname{pm}}(p)$ & $52.6537$ & $27.7971$ & $15.4647$ & $13.1253$ & $11.9286$ & $11.6684$ & $11.3592$ & $11.2589$  \\
    \hline
        $\log_p M^{\operatorname{pp}}(p)$ & $100.137$ & $54.3723$ & $30.8617$ & $26.1284$ & $23.2982$ & $22.5834$ & $21.6647$ & $21.3477$  \\
    \hline
    \hline
        $\log_p M^\nw(p)$ & $250.811$ & $141.545$ & $84.5551$ & $71.1082$ & $61.7564$ & $59.2759$ & $55.9947$ & $54.8340$  \\
    \hline
        $\log_p M^{\nw, \operatorname{pp}}(p)$ & $101.722$ & $55.3724$ & $31.5443$ & $26.6930$ & $23.7564$ & $23.0117$ & $22.0525$ & $21.7209$  \\
    \hline
    \hline
    $\log_p M^{\operatorname{AL}}(p)$ & $1554.47$ & $929.611$ & $597.117$ & $497.170$ & $415.545$ & $393.093$ & $362.847$ & $351.995$  \\
    \hline
    $\log_p M^{\nw, \operatorname{AL}}(p)$ & $1585.39$ & $949.124$ & $610.442$ & $508.190$ & $424.486$ & $401.451$ & $370.413$ & $359.275$  \\
    \hline
    $\log_p M^{\operatorname{AL}, \operatorname{pm}}(p)$ & $168.208$ & $93.7801$ & $55.1558$ & $46.3924$ & $40.4985$ & $38.9565$ & $36.9339$ & $36.2238$  \\
    \hline
    \end{tabular}
\end{center}
\end{proposition}

\begin{remark}\label{rmk:C-decreasing-modification}
    While $C_p(X)$ and its variants $C_p^{\text{condition}}(X)$ do not necessarily decrease in $p$ for fixed values of $X= \log N$, the functions $\log_p M(p)$ and $\log_p M^{\text{condition}}(p)$ do seem to decrease. This fact is difficult to prove formally in general but easy to exploit. For example, it appears from the above table that $\log_p M(p)< 30$ for all $p \geq 17.$ This can be verified by replacing the function $C_p(X)$ with a slightly larger function $C(X)$ independent of $p$ such that
    \begin{enumerate}
        \item $C_p(X) \leq C(X)$ for large $p,$
        \item $C(X)$ is decreasing for large enough $X.$
    \end{enumerate}
    For large enough $p$, one then computes an upper bound for $\log_p M(p)$ via Corollary \ref{cor:E(p)-relation} using $C(X)$ in place of $C_p(X)$; this bound is guaranteed to decrease in $p.$ Therefore, it only takes a finite check to verify that $\log_p M(p) < 30$ for all $p \geq 17.$
\end{remark}

Using the approach in Remark \ref{rmk:C-decreasing-modification}, the following bounds then follow from a finite computation.

\begin{corollary} \label{cor:M(p)-for-sufficiently-large-p}~ \\
    If $J_0(N)$ is supersingular modulo $p,$ then we have the following bounds.
    \begin{enumerate}
        \item If $p \geq 97$, then $N \leq p^{25}$.
        \item If $p \geq 45486901$, then $N \leq p^{20}$.
        \item If $p \geq 29,$ all prime divisors of $N$ are less than $p^{11}$.
        \item If $p \geq 37,$ all prime power divisors of $N$ are less than $p^{20}$.
        \item If $p \geq 12451,$ all prime power divisors of $N$ are less than $p^{17}$.
    \end{enumerate}
    If $J_0^\nw(N)$ is supersingular modulo $p,$ then we have the following bounds.
    \begin{enumerate}
        \item If $p \geq 569$, then $N \leq p^{40}$.
        \item If $p \geq 5072004923$, then $N \leq p^{30}$.
        \item If $p \geq 47$ and $N$ is a prime power, then $N \leq p^{20}$.
        \item If $p \geq 24281$ and $N$  is a prime power, then $N \leq p^{17}$.
    \end{enumerate}
    If $J_0^\sigma(N)$ is supersingular modulo $p,$ then we have the following bounds.
    \begin{enumerate}
        \item If $p \geq 41,$ then $N \leq p^{300}$.
        \item If $p \geq 2803309010977,$ then $N \leq p^{100}$.
        \item If $p \geq 113$ and $N$ is prime, then $N \leq p^{30}.$
    \end{enumerate}
    If $J_0^{\nw, \sigma}(N)$ is supersingular modulo $p,$ then we have the following bounds.
    \begin{enumerate}
        \item If $p \geq 43,$ then $N \leq p^{300}$.
        \item If $p \geq 5159783983903,$ then $N \leq p^{100}$.
    \end{enumerate}
\end{corollary}

\begin{proof}
    All these results follow from finite computations as described in Remark \ref{rmk:C-decreasing-modification}. We note that the statements about prime and prime power divisors of $N$ follow from the fact that
    $$\max \left\{ \text{$P$ prime}  \mid \text{$p \in \operatorname{SS}\!\left(J_0(N) \right)$ for some $N$ a multiple of $P$}  \right\} \leq M^{\text{pm}}(p),$$
    and likewise for $C_p^{\text{pp}}.$ This is since for any divisor $P$ of $N,$ $J_0(P)$ is a factor of $J_0(N).$
\end{proof}

Furthermore, Corollary \ref{cor:E(p)-relation} implies that
\begin{align}
    \limsup_{p \rightarrow \infty} \, \log_pM(p) &\leq 32 \, \lim_{p \rightarrow \infty} \lim_{N \rightarrow \infty} C_p(\log N) \qquad \text{by Corollary \ref{cor:E(p)-relation} and since $C_p(X) \rightarrow \lim_{X \rightarrow \infty}$}\\
    &\qquad \text{converges uniformly in $p$ by the explicit forms of $\{C_p\}_p$ in Appendix \ref{section:explicit-form-of-function}}\\
    &= 32 \, \lim_{p \rightarrow \infty} \frac{3}{2} \| \mu_p \| \qquad \text{by \eqref{eqn:Atkin-Lehner-equidistribution}}\\
    &= 32 \, \lim_{p \rightarrow \infty} \frac{3}{2} \cdot \frac{1}{\pi} \qquad \text{by Definition \ref{def:mu-p}},
\end{align}
and analogously for $M^{\text{condition}}(p),$ with $C_p$ replaced with $C_p^{\text{condition}}$.

These bounds yield the following Corollary.
\begin{corollary} \label{cor:vertical-suprsingularity-with explicit exponent}
    As $p$ grows, 
    \begin{enumerate}
        \item If $J_0(N)$ is supersingular modulo $p,$ then $N \leq p^{\frac{48}{\pi} + o(1)}$ and $P \leq p^{\frac{32}{\pi} + o(1)}$ for any prime $P \mid N.$
        \item If $J_0^\nw(N)$ is supersingular modulo $p,$ then $N \leq p^{\frac{48}{\pi} + o(1)}$.
    \item If $J_0^\sigma(N)$ is supersingular modulo $p,$ then $N \leq p^{\frac{64}{\pi} + o(1)}$. 
    \item If $J_0^{\nw,\sigma}(N)$ is supersingular modulo $p,$ then $N \leq p^{\frac{64}{\pi} + o(1)}$.
    \end{enumerate}
\end{corollary}

\section{Horizontal supersingularity}\label{section:horizontal}

\subsection{Main horizontal results}

In \S\ref{subsection:horizontal-results} we defined, for $f \in S_2^{\nw}(N),$
\begin{align}
    V(f) &:= \left\{ \text{$p$ prime} : a_p(f) = 0  \right\}.\\
        V(N) &:= \left\{ \text{$p$ prime} : a_p(f) = 0 \quad \forall f \in S_2^\nw(N) \right\}.
\end{align}

We now prove Proposition \ref{prop:structure-result}.
{
\renewcommand{\thetheorem}{\ref{prop:structure-result}}
\begin{proposition}
    For any newform $f \in S_2^{\nw}(N),$ 
    \begin{align}
        \operatorname{SS}(A_f) \subseteq \left\{ p : p \mid \operatorname{disc}(K_f) \text{ or } p - 1 \mid 4[K_f : \mathbb{Q} ] \right\} \cup V(f). \label{eqn:RHS-structure-result}
    \end{align}
    In particular, $\operatorname{SS}(A_f) \setminus V(f)$ is a finite set.
\end{proposition}
\addtocounter{theorem}{-1}
}

\begin{proof}
Suppose that $A_f$ is supersingular modulo $p$, and that $p \notin V(f)$. Then write 
$a_p(f) = \alpha + \overline\alpha$, where $\alpha,\overline \alpha$ is a certain conjugate pair of $p$-Frobenius eigenvalues for $A_f$. Since $A_f$ is supersingular modulo $p$, $\alpha$ is of the form $\alpha = \sqrt p \, \zeta_n$, where 
$\zeta_n$ is a primitive $n$-th root of unity, and so $\zeta_n + \overline{\zeta_n} = \frac{1}{\sqrt p} \alpha + \frac{1}{\sqrt p} \overline{\alpha} = \frac{1}{\sqrt{p}} a_p(f) \ne 0$. This then implies that
$$\sqrt{p} = \frac{a_p(f)}{\zeta_n + \overline{\zeta_n}} \in K_f(\zeta_n),$$
so $p$ is ramified in  $K_f(\zeta_n)$. Hence $p$ is either ramified in $K_f$ or ramified in $\mathbb{Q}(\zeta_n).$ In the first case, $p \mid \operatorname{disc}(K_f).$ In the second case, $p \mid n$ by \cite[Proposition 2.3]{washington1997introduction}, so 
\begin{align}
    p - 1 &\ \ \text{divides} \ \ \varphi(n) = \left[ \mathbb{Q}(\zeta_n) : \mathbb{Q}  \right]\\
    &\ \  \text{divides} \ \   2 \cdot \left[ \mathbb{Q} \left( \zeta_n + \overline{\zeta_n}\right) : \mathbb{Q} \right] = 2 \cdot  \left[ \mathbb{Q} \left( \frac{a_p(f)}{\sqrt{p}} \right) : \mathbb{Q} \right] \\
    &\ \  \text{divides} \ \  2 \cdot  [K_f (\sqrt{p}) : \mathbb{Q}] \\
    &\ \  \text{divides} \ \  4 \cdot [K_f : \mathbb{Q}],
\end{align}
completing the proof.
\end{proof}

Proposition \ref{prop:structure-result} has several consequences. Firstly, given a newform $f \in S_2^\nw(N),$ the constant
\begin{align}
    C_f := \max \left\{ \operatorname{disc}(K_f), 4[K_f : \mathbb{Q}] + 1\right\} \label{eqn:constant-C-f}
\end{align}
satisfies Theorem \ref{main-horizontal}, which we restate here.

{
\renewcommand{\thetheorem}{\ref{main-horizontal}}
\begin{theorem}
    Fix an abelian variety $A_f$. Then for all $p > C_f,$ $A_f$ is supersingular modulo $p$ if and only if it is isogenous to a power of the unique (up to isogeny) supersingular elliptic curve modulo $p$. 
\end{theorem}
\addtocounter{theorem}{-1}
}

\begin{proof}
Suppose that $A_f$ is supersingular modulo $p.$ 
Then since $p > C_f$, Proposition \ref{prop:structure-result} implies that $a_p(f) = 0$. Hence, the $p$-Weil polynomial of $A_f$ is $P_{A_f}(x) = (x^2+p)^{[K_f:\mathbb{Q}]},$ so $A_f$ is isogenous to a power of the unique supersingular elliptic curve (since $p > C_f  \geq 5$). The converse is clear.
\end{proof}

We also have the following corollaries.

\begin{corollary}\label{superspecial}
Fix an abelian variety $A_f$. Then for all $p > C_f,$ $A_f$ is superspecial modulo $p$ if and only if it is isogenous to a power of the unique (up to isogeny) supersingular elliptic curve modulo $p$. 
\end{corollary}

\begin{proof}
By \cite[Lemma 2.2]{yu2011superspecial}, if $A \sim_{\mathbb{F}_p} E_0^g,$ where $E_0$ is the unique supersingular elliptic curve modulo $p$, then $A_p$ is superspecial.
\end{proof}

\begin{corollary}\label{cor:horizontal-result-for-Jac}
    Fix $N \ge 1$. Then for all $p > \max_{f \in S_2^{\nw}(N)} C_f,$ $J_0^{\nw}(N)$ is supersingular modulo $p$ if and only if
    it is isogenous to a power of the unique (up to isogeny) supersingular elliptic curve modulo $p$. 
\end{corollary}

\begin{proof}
    Recall that $J_0^{\nw}(N) \sim_{\mathbb{Q}} \prod_{f} A_f \label{eqn:decomp-J_0-new-mod-p},$ where the product over $f$ runs over Galois orbits of newforms $f \in S_2^{\nw}(N)$. Then, $J_0^{\nw}(N)$ is supersingular modulo $p$ if and only if each $A_f$ is supersingular modulo $p.$ If $p > \max_{f \in S_2^{\nw}(N)} C_f,$ then Theorem \ref{main-horizontal} implies that each $A_f$ is isogenous to a power of the unique supersingular elliptic curve over $\mathbb{F}_p$. The same is then true of $J_0^\nw(N)$.
\end{proof}

\begin{corollary}\label{cor:horizontal-result-for-Jac-full}
    Fix $N \ge 1$. Then for all $p > \max_{f \in  \{S^\nw_2(M),\, M \mid N\}} \{C_f\}$, $J_0(N)$ is supersingular modulo $p$ if and only if
    it is isogenous to a power of the unique (up to isogeny) supersingular elliptic curve modulo $p$. 
\end{corollary}

\begin{proof}
    This follows from Corollary \ref{cor:horizontal-result-for-Jac} and the decomposition $J_0(N) \sim_{\mathbb{Q}} \prod_{M \mid N} J_0^{\nw}(M)^{\sigma_0(N/M)}$. 
\end{proof}

We now prove Theorem \ref{thm:Hui-improvement}, restated here.

{
\renewcommand{\thetheorem}{\ref{thm:Hui-improvement}}
\addtocounter{theorem}{-1}
\begin{theorem}
    There exist universal constants $c_1$ and $c_2$ such that for any $X \geq 3$,
    \begin{enumerate}
        \item For each non-CM newform $f \in S_2^\nw(N),$
    \begin{align}
        \#  \operatorname{SS}(A_f) \cap \{p : p \leq X\} \leq c_1 \cdot \frac{\log (2N \log X)}{\sqrt{\log X}} \cdot \pi(X) + \frac{N}{3}.
    \end{align}
    \item For each non-CM newform $f \in S_2^\nw(N),$ if $L(s, \operatorname{Sym}^m (f))$ satisfies GRH for all $m \geq 0,$
    \begin{align}
        \#  \operatorname{SS}(A_f) \cap \{p : p \leq X\} \leq c_2 \cdot \frac{\log (2N \log X)}{X^{1/4}} \cdot \pi(X) + \frac{N}{3}.
    \end{align}
    \end{enumerate}
    Here, $\pi(X)$ denotes the prime counting function.
\end{theorem}
}
    
\begin{proof}
    By Proposition \ref{prop:structure-result}, it suffices to bound the intersection of $\{1, 2, \cdots, \lfloor X \rfloor \}$ with the three components of the right hand side of \eqref{eqn:RHS-structure-result}. Firstly, Corollary \ref{cor:omega-log-disc}, to be proved in \S\ref{subsection:discriminant-bound}, gives that
    \begin{align}
        \#\left\{ p : p \mid \operatorname{disc}(K_f) \text{ or } p - 1 \mid 4[K_f : \mathbb{Q} ] \right\} \leq \frac{N}{3}.
    \end{align} 
    
    Next, we use Thorner's effective Sato-Tate result \cite[Theorem 1.1]{Thorner2021Effective} to bound the third term:
    \begin{align}
        \# \{p \leq X : p \in V(f)\} &\leq \pi(x) \cdot \int_{0}^0 \, d \mu_{ST} + c_1 \, \pi(X) \cdot \frac{\log (2N \log X)}{\sqrt{\log X}} \qquad \text{for some absolute $c_1$}\\
        &= c_1 \, \pi(X) \cdot \frac{\log (2N \log X)}{\sqrt{\log X}}, \qquad \text{since $\mu_{ST}(\{0\}) = 0.$}
    \end{align}
    This proves (1). Part (2) follows likewise, using \cite[Theorem 1.3]{Thorner2021Effective} in place of \cite[Theorem 1.1]{Thorner2021Effective}.
\end{proof}

An analogous result holds for $\operatorname{SS}(J_0^{\nw}(N))$ for all but finitely many $N.$

\begin{corollary}\label{cor:Hui-improvement-for-Jacobian}
    Suppose $N \notin \{ 27, 32, 36, 49, 64, 108, 256 \}.$ Then
    \begin{enumerate}
        \item The Jacobian $J_0^{\nw}(N)$ satisfies
    \begin{align}
        \#  \operatorname{SS}(J_0^{\nw}(N)) \cap \{p : p \leq X\} \leq c_1 \cdot \frac{\log (2N \log X)}{\sqrt{\log X}} \cdot \pi(X) + \frac{N}{3}.
    \end{align}
    \item If there exists a non-CM newform $f \in S_2^\nw(N),$ such that $L(s, \operatorname{Sym}^m (f))$ satisfies GRH for all $m \geq 0,$ then
    \begin{align}
        \#  \operatorname{SS}(J_0^{\nw}(N)) \cap \{p : p \leq X\} \leq c_2 \cdot \frac{\log (2N \log X)}{X^{1/4}} \cdot \pi(X) + \frac{N}{3}.
    \end{align}
    \end{enumerate}
    where $c_1$ and $c_2$ are the universal constants in Theorem \ref{thm:Hui-improvement}.
\end{corollary}

\begin{proof}
    The forthcoming work \cite{Mondal2026} shows that for all levels $N \notin \{ 27, 32, 36, 49, 64, 108, 256 \},$ there exists at least one non-CM newform $f \in S_2^{\nw}(N).$ For this newform $f$, $A_f$ is a factor of $J_0^\nw(N).$ The result then follows from Theorem \ref{thm:Hui-improvement}.
\end{proof}

In light of Corollary \ref{cor:Hui-improvement-for-Jacobian}, it is natural to ask about the behavior of supersingular primes for $J_0^\nw(N)$ when $N \in \{27,32,36,49,64,108,256\}$. For these seven levels, it turns out that a positive proportion of primes are supersingular.
\begin{proposition} \label{prop:seven-exceptional-levels}
    For $N \in \{27,32,36,49,64,108,256\}$, let
    \begin{align}
        \delta_N := 
        \begin{cases}
            \frac{1}{2} & \text{if } N \in \{27,32,36,49,64,108\}, \\
            \frac{1}{4} & \text{if } N=256.
        \end{cases}
    \end{align}
    Then as $X \to \infty$,
    \begin{align}
        \# (\operatorname{SS}(J_0^\nw(N)) \cap \{p : p \le X\}) \sim \delta_N \,\pi(X).
    \end{align}
\end{proposition}
\begin{proof}
    Recall that 
    \begin{align}
        J_0^{\nw}(N) \sim_{\mathbb{Q}} \prod_{f} A_f,
    \end{align}
    where $f$ runs over Galois orbits of newforms in $S_2^{\nw}(N)$. 
    For $N \in \{27,32,36,49,64,108,256\}$, every such newform $f$ has CM by an imaginary quadratic field $M_f$. Then, we have
    \begin{align} \label{eqn:AfL-isogenous-power-E}
        A_f \sim_{\overline{\mathbb{Q}}} E^{\dim A_f}
    \end{align}
    for an elliptic curve $E/{\overline{\mathbb{Q}}}$ for which $\operatorname{End}^0_{\overline{\mathbb{Q}}}(E) =M_f$ (e.g. see \cite[proof of Theorem~4.7]{ribet1980twists}). Since every elliptic curve with CM by $M_f$ is defined over a number field $L$, we in fact have $(A_f)_L \sim_{L} E^{\dim A_f}$, where we are considering the base changes of $A_f$ and $E$ to $L$.
    Let $p \nmid N$, and choose a prime $\mathfrak p$ of $L$ above $p$. Since $A_f$ has good reduction at $p$, its base change $(A_f)_L$ has good reduction at $\mathfrak p$, and hence so does $E$. Moreover, the isogeny between the two sides of \eqref{eqn:AfL-isogenous-power-E} induces an isogeny between their reductions over a finite extension of $\mathbb{F}_p$ (specifically, the residue field of $\mathfrak p$). Hence since supersingularity is invariant under isogeny and finite extensions of the finite ground field, $A_f$ is supersingular modulo $p$ if and only if $E$ is supersingular modulo $\mathfrak p$.

    Note that in each of the cases listed below, the discriminant of $M_f$ divides $N$, so each $p \nmid N$ is unramified in $M_f$. Hence Deuring's reduction theorem implies that for $p \nmid N$, 
    $A_f$ is supersingular modulo $p$ 
    if and only if     $p$ is inert in $M_f$. In particular, this means that for $p \nmid N$, $J_0^\nw(N)$ is supersingular modulo $p$ if and only if $p$ is inert in $M_f$ for each factor $A_f$ in the $\QQ$-isogeny decomposition of $J_0^\nw(N)$.
    
    In the following table, for each $N \in \{27,32,36,49,64,108,256\}$, we list the distinct CM fields $M_f$ for the factors $A_f$ in the $\QQ$-isogeny decomposition of $J_0^\nw(N)$; this data comes from the LMFDB \cite{lmfdb}.
    We also list the primes that are simultaneously inert in each of these CM fields.

\begin{equation}
\begin{array}{c|c|c}
N & \text{CM fields} & \text{primes simultaneously inert in these CM fields} 
\\
\hline
27  & \QQ(\sqrt{-3}) & \lrp{\frac{-3}{p}}=-1 \iff p \equiv 2 \pmod 3   \\
32  & \QQ(i)         & \lrp{\frac{-1}{p}}=-1 \iff p \equiv 3 \pmod 4 \\
36  & \QQ(\sqrt{-3}) & \lrp{\frac{-3}{p}}=-1 \iff p \equiv 2 \pmod 3 \\
49  & \QQ(\sqrt{-7}) & \lrp{\frac{-7}{p}}=-1 \iff p \equiv 3,5,6 \pmod 7 \\
64  & \QQ(i)         & \lrp{\frac{-1}{p}}=-1 \iff p \equiv 3 \pmod 4 \\
108 & \QQ(\sqrt{-3}) & \lrp{\frac{-3}{p}}=-1 \iff p \equiv 2 \pmod 3 \\
256 & \QQ(i),\QQ(\sqrt{-2})  & \lrp{\frac{-1}{p}}=\lrp{\frac{-2}{p}}=-1 \iff p \equiv 7 \pmod 8 \\
\end{array}
\end{equation}
    Finally, by the prime number theorem in arithmetic progressions, the equivalence conditions given in the last column of this table yield that the proportion $\delta_N$ of supersingular primes for $J_0^\nw(N)$ is given by
    \begin{align}
        \delta_N = 
        \begin{cases}
            \frac{1}{2} & \text{if } N \in \{27,32,36,49,64,108\}, \\
            \frac{1}{4} & \text{if } N=256,
        \end{cases}
    \end{align}
    as desired.
\end{proof}

Now, recall the following generalized Maeda conjecture (originally \cite[Conjecture A]{Kimball2021}, statement from \cite[Conjecture 4.14]{KMRX}).

\begin{conjecture}[Generalized Maeda conjecture for weight $2$] \label{conj:gen-Maeda}
    For 100\% of $N$, the characteristic polynomial of the Hecke operator $\mathbf{T}_p$ over $S_2^{\nw, \sigma}(N)$ is irreducible, where $\sigma$ any admissible sign pattern for $N$.
 \end{conjecture}

Assuming this conjecture, Theorem \ref{thm:Hui-improvement} can be improved to the assertion that $\operatorname{SS}(A_f)$ and $\operatorname{SS}(J_0^\nw(N))$ are finite, not just density zero.

\begin{corollary}\label{cor:finiteness-assuming-Maeda}
    Assume Conjecture \ref{conj:gen-Maeda}. Then for 100\% of levels $N$, the following is true for any newform $f \in S_2^{\nw}(N)$.
    \begin{enumerate}
        \item $A_f$ is not supersingular modulo any prime greater than $C_f.$
        \item $\# \operatorname{SS}(A_f) \leq \frac{N}{3}.$
    \end{enumerate}
    The same conclusions also hold for 100\% of $J_0^\nw(N),$ with $C_f$ replaced by $\min_{f \in S_2^\nw(N)} C_f$.
\end{corollary}

\begin{proof}
    Consider the density $1$ set of $N$ such that
    \begin{enumerate}
        \item[(a)] The conclusion of Conjecture \ref{conj:gen-Maeda} holds,
        \item[(b)] $\dim S_2^{\nw, \sigma}(N) \geq 2$ for all (admissible) Atkin-Lehner sign patterns $\sigma$ (which holds for all but finitely many $N$ by \cite[Theorem 1.1]{RVWX26}).
    \end{enumerate}
    Then, for each newform $f \in S_2^\nw(N)$, let $\sigma$ be its associated sign pattern. By condition (a), the characteristic polynomial of $\mathbf{T}_p$ over $S_2^{\nw, \sigma}(N)$ is irreducible for all primes $p$. In particular, by condition (b), this means that $0$ cannot a root the characteristic polynomial of $\mathbf{T}_p$ for all primes $p$, and so $V(f)=\varnothing$. Consequently, for $p > C_f$, $a_p(f) \ne 0$ implies that $A_f$ cannot be supersingular modulo $p$ (by Proposition \ref{prop:structure-result}), verifying part (1). The fact that $V(f) = \varnothing$ also proves (2) by the proof of Theorem \ref{thm:Hui-improvement}.
\end{proof}

\subsection{A bound on \texorpdfstring{$\operatorname{disc}(K_f)$}{disc(Kf)}}\label{subsection:discriminant-bound}

In this subsection, we prove Corollary \ref{cor:omega-log-disc}, which we use in the proof of Theorem \ref{thm:Hui-improvement}. We follow the strategy used in \cite[Lemma 4.1]{Bruin2011ComputingCoefficients} to bound the discriminant of the Hecke algebra $\mathbb{T}(S_2(N))$.

\begin{lemma}\label{lem:disc-bound}
Let $f \in S_2^{\mathrm{new}}(N)$ be a newform and let $K_f$ be its Hecke field. Then
\[
\log \bigl|\operatorname{disc}(K_f)\bigr| \;\le\; \frac{2}{9} \,  N  \log N.
\]
\end{lemma}

\begin{proof}
We prove the bound for newforms of level $N \geq 66.$ For $N \le 65,$ it suffices to check this bound by computation; see the accompanying computational code in \cite{code2}.

Let $r_f = [K_f : \mathbb{Q}]$. By \cite[Theorem 2]{martin2005dimensions}, we have
\begin{align}
    r_f &\leq \dim S_2^\nw(N) \le \frac{N + 1}{12}. \label{eqn:bound-for-Hecke-degree}
\end{align}
Additionally, by the Sturm bound \cite[Theorem 2.3]{AgasheStein2005} and \eqref{eqn:Sole-Planat-psi-bound}, the Hecke algebra $\mathbb{T} := \mathbb{T}\bigl(S_2^{\mathrm{new}}(N)\bigr)$ can be generated as a $\mathbb{Z}$-module by the
Hecke operators $\mathbf{T}_m$ for indices up to
\begin{align}
    d := \left\lceil \frac{\psi(N)}{6}\right\rceil  \leq \frac{e^\gamma}{6} \, N \log \log N + 1.
\end{align}
Consider the maps $\lambda_f : \mathbb{T} \to \mathcal{O}_{K_f}$ sending a Hecke operator $\mathbf{T}_m$
to the Fourier coefficient $a_m(f) \in \mathcal{O}_{K_f}$. Since
$\mathbb{T}$ is generated as a $\mathbb{Z}$-module by $\mathbf{T}_1,\dots,\mathbf{T}_d$, the image $\lambda_f(\mathbf{T})$ is
generated as a $\mathbb{Z}$-module by $a_1(f),\dots,a_d(f)$. Since $K_f = \mathbb{Q}(a_1(f),
a_2(f),\dots)$, this image spans $K_f$ over $\mathbb{Q}$. We may
select indices $1 \le m_1 < m_2 < \cdots < m_{r_f} \le d$ such that the elements
\[
\omega_u := \lambda_f(T_{m_u}) = a_{m_u}(f) \in \mathcal{O}_{K_f} \qquad (u=1,\dots,r_f)
\]
are $\mathbb{Q}$-linearly independent and hence form a $\mathbb{Q}$-basis of $K_f$. Now, let $\Lambda_f$ be the $\mathbb{Z}$-sublattice of $\mathcal{O}_{K_f}$ spanned by
$\{\omega_1,\dots,\omega_{r_f}\}$. Since $\Lambda_f$ is a full-rank $\mathbb{Z}$-submodule of
$\mathcal{O}_{K_f}$, we have the relation
\begin{align}
    \left| \operatorname{disc}(K_f) \right| &\leq \left| \operatorname{disc}(\Lambda_f) \right|\\
    &= \det\lrb{ \operatorname{Tr}_{K_f/\mathbb{Q}}(\omega_u \omega_v) }_{u,v=1}^{r_f}\\
    &= \det\lrb{ \sum_{t=1}^{r_f} \tau_t(a_{m_u}(f))\,\tau_t(a_{m_v}(f)) }_{u,v=1}^{r_f} 
    \qquad \text{for $\{ \tau_1, \cdots, \tau_{r_f}\}: K_f \hookrightarrow \mathbb{C}.$}\\
    &= \det \left( A^T A  \right)  
    \qquad \text{where $A = {\Big[ \tau_t \left( a_{m_u}(f) \right)  \Big]_{t, u = 1}^{r_f}}$}\\
    &= \det\left( \Big[ \tau_t \left( a_{m_u}(f) \right)  \Big]_{t, u = 1}^{r_f} \right)^2\\
    &\le
        \left( \sup_{1\le t, u \le r_f} 
        |\tau_t \left( a_{m_u}(f) \right)|  \right)^{2 r_f} r_f^{r_f}
    \qquad \text{(by Hadamard's inequality \cite[Theorem 1]{browne2021survey})} \\
    &\leq \left( \sup_{1\le u \le r_f} \sigma_0(m_u)^2 \cdot m_u  \right)^{r_f} \cdot r_f^{r_f} \qquad \text{(by Deligne's bound)} \\
    &\leq \left( 2^{3.07588 \cdot \frac{\log d}{\log \log d}} \cdot d  \right)^{r_f} \cdot r_f^{r_f} \qquad\qquad \text{(by \eqref{eqn:nicolas-robin-sigma-bound} since $m_u \le d$)}.
\end{align}
Taking logarithms, we obtain
\begin{align}
    &\log \left| \operatorname{disc}(K_f) \right| \\
    &\leq r_f \left( \log r_f + \log d + \frac{\log d}{\log \log d} \cdot 3.07588 \, \log 2 \right)\\
    &\leq  \frac{N + 1}{12} \!\left( \log \left( \frac{N + 1}{12} \right) + \log \left( \frac{e^\gamma}{6} N \log \log N + 1 \right) + \frac{\log \left( \frac{e^\gamma}{6} N \log \log N + 1 \right)}{\log \log \left( \frac{e^\gamma}{6} N \log \log N + 1 \right)} \cdot 3.07588 \, \log 2 \right)\\
    &\leq \frac{2}{9} \,  N   \log N    \qquad \text{for $N \geq 66$}.
\end{align}
This completes the proof.
\end{proof}

Lemma \ref{lem:disc-bound} implies following result, which we use in the proof of Theorem \ref{thm:Hui-improvement}.

\begin{corollary}\label{cor:omega-log-disc}
    Let $f \in S_2^\nw(N)$ be a newform and let $K_f$ be its Hecke field. Then
    \begin{align}
        \#\left\{ p : p \mid \operatorname{disc}(K_f) \text{ or } p - 1 \mid 4[K_f : \mathbb{Q} ] \right\} \leq \frac{N}{3}.
    \end{align}
\end{corollary}

\begin{proof}
    We prove the statement for $N \geq 224$. All other cases can be checked by computation; see the accompanying computational code in \cite{code2}. We have
    \begin{align}
        \omega\left( \operatorname{disc}(K_f) \right) &\leq 1.38407 \cdot \frac{\log \operatorname{disc}(K_f)}{\log \log \operatorname{disc}(K_f)} \qquad \text{by \eqref{eqn:robin-omega-bound}}\\
        &\leq 1.38407 \cdot \frac{\frac{2}{9} \, N  \log N}{\log \left( \frac{2}{9} \, N  \log N \right)} \qquad \text{by Lemma \ref{lem:disc-bound}}\\
        &\leq 0.30758 \,  N \qquad \text{for $N \geq 90$}.
    \end{align}
    Next, we have
    \begin{align}
        \#\{ p : p - 1 \mid 4[K_f: \mathbb{Q}]\} &\leq \sigma_0(4[K_f: \mathbb{Q}])\\
        &\leq 2^{1.53794 \cdot \frac{\log(4[K_f: \mathbb{Q}]) }{\log \log (4[K_f: \mathbb{Q}])}} \qquad \text{by \eqref{eqn:nicolas-robin-sigma-bound}}\\
        &\leq 2^{1.53794 \cdot \frac{\log((N + 1)/3) }{\log \log ((N + 1)/3)}} \qquad \text{by \eqref{eqn:bound-for-Hecke-degree}}\\
        &\leq 0.025752 \, N \qquad \text{for $N \geq 1445$},
    \end{align}
    although computation proves that this bound also holds for all newforms of level $N \in [224, 1444]$; see the accompanying computational code in \cite{code1}. The fact that $0.30758 + 0.025752 < \frac{1}{3}$ suffices to complete the proof.
\end{proof}

\appendix

\section{Explicit forms of the function \texorpdfstring{$C_p(X)$}{Cp(X)} and analogs}\label{section:explicit-form-of-function}

In this appendix, we establish explicit formulas for $C_p(X)$ and its analogs. We rely on results and techniques from \cite{KMRX} and \cite{MS07}. We provide full details for the derivation of $C_p(X)$ in Appendix \ref{subsection:explicit-form-C}. The derivations of the other analogs are similar, so we just write down the key steps. Throughout this appendix, we let $W$ denote the lambert $W$ function, which is the inverse function of $x \mapsto xe^x.$

\subsection{Explicit form of \texorpdfstring{$C_p(X)$}{Cp(X)}}\label{subsection:explicit-form-C}

Murty and Sinha give in \cite[Theorem 19]{MS07} that for any $M \in \mathbb{N},$
\begin{align}
    &\left|
\frac{\# \operatorname{eigen}_{S_2(N)}(\mathbf{T}'_p) \cap I}
{\dim S_2(N)}
-
\int_{-2}^2 \chi_I(x)\, d\mu_p(x)
\right| \\
&\leq \frac{\| \mu_p\|}{M + 1} + \frac{4p^M 2^{\omega(N)} \sup_{n^2 < 4p^M} \psi(n) + 2\Sigma(N) +  2p^{M/2}}{\dim S_2(N)}, \label{eqn:Murty-Sinha-Theorem-19}
\end{align}
where $ \Sigma(N) := \sum_{c \mid N} \phi(\gcd(c, N/c) ).$ By \cite[Page 696]{MS07}, $\Sigma(N) \leq\sqrt{N} \cdot \sigma_0(N)$.

\begin{remark}
    Murty and Sinha state the bound \eqref{eqn:Murty-Sinha-Theorem-19} with the term $\frac{1}{M + 1}$ rather than $\frac{\| \mu_p\|}{M + 1},$ this being sufficient for their purposes. However, their proof reveals that the latter term also holds.
\end{remark}

By \cite[Corollary 15]{MS07}, $\dim S_2(N) \ge \frac{3}{200} N$ for $N \geq 61.$ Therefore,
\begin{align}
    &\eqref{eqn:Murty-Sinha-Theorem-19}\\
    &\leq \frac{\| \mu_p\|}{M + 1} + \frac{4p^{M} 2^{\omega(N)}\cdot 2p^{M/2} \cdot e^\gamma \log \log \left( 2 p^{M/2} \right) + 2\sqrt{N}\sigma_0(N) + 2p^{M/2}}{\frac{3}{200} N} \quad\ \  \text{by \eqref{eqn:Sole-Planat-psi-bound}}\\
    & \leq \frac{\| \mu_p\| \log p}{c \log N} + \frac{1600 e^\gamma \cdot \sigma_0(N)}{3} \cdot  \frac{N^{3c/2} \log \log \left( 2N^{c/2}  \right) + 2\sqrt{N} + N^{c/2}}{ N} \quad\ \text{for $M = \left\lfloor \frac{c \log N}{\log p} \right\rfloor$}\\
    & \leq \frac{\| \mu_p\| \log p}{c \log N} + \frac{1600 e^\gamma \cdot \sigma_0(N)}{3} \cdot  \frac{  N^{3c/2} \log \log N + 2\sqrt{N} + N^{c/2}}{N} \qquad\qquad \text{for $c \leq \frac{2}{3}$}\\
    & \leq \frac{\log p}{\log N} \left( \frac{\| \mu_p\|}{c} + \frac{1600 e^\gamma \cdot \sigma_0(N) \cdot \log N}{3 \log p} \cdot  \frac{  N^{3c/2} \log \log N + 2\sqrt{N} + N^{c/2}}{N} \right)\\
    & \leq \frac{\log p}{\log N} \left( \frac{\| \mu_p\|}{c} + \frac{1508.49}{\log p} \cdot  \frac{N^{3c/2} \cdot \sigma_0(N) \cdot \log N \cdot \log \log N}{N} \right) \qquad \text{for $c \geq \frac{1}{3}$ and $N \geq e^{30}$}\\
    & = \frac{\log p}{\log N} \left( \| \mu_p \| \left( \frac{2}{3} + \frac{2\log K}{3X} - \frac{2 \, \log \left( X \log X  \right)}{3 X} - \frac{2 \log \sigma_0(N)}{3 X}  \right)^{-1} + \frac{1508.49}{\log p} \cdot K  \right)\\
    & \qquad \text{after substituting $X = \log N$ and $c = \frac{2}{3} + \frac{2\log K}{3X} - \frac{2\log (X \log X)}{3X} - \frac{2 \log \sigma_0(N)}{3X}$}\\
    &\leq \frac{\log p}{\log N} \left( \| \mu_p \| \left( \frac{2}{3} + \frac{2\log K}{3X} - \frac{2 \, \log \left( X \log X  \right)}{3 X} - \frac{2 \cdot 1.06602}{3 \log X}  \right)^{-1} + \frac{1508.49}{\log p} \cdot K  \right) \qquad \text{by \eqref{eqn:nicolas-robin-sigma-bound}}.
\end{align}
    A technique for choosing the value of $K,$ dependent on $N,$ that minimizes functions like this is given in \cite[Appendix A]{KMRX} and yields a minimizing $K$ of
    \begin{align}
        K := \frac{3 \| \mu_p \| \log p}{8 \cdot 1508.49} \, X \cdot  \, W \! \left( \sqrt{\frac{3 \| \mu_p \| \log p}{8 \cdot 1508.49 \log X}} \exp \left( \frac{X}{2} - \frac{1.06602 \, X }{2 \log X} \right) \right)^{-2}.
    \end{align}
    We can thus define
    \begin{align}
        \quad C_p(X) := \| \mu_p \| \left( \frac{2}{3} + \frac{2\log K}{3X} - \frac{2 \, \log \left( X \log X  \right)}{3 X} - \frac{2 \cdot 1.06602}{3 \log X}  \right)^{-1} + \frac{1508.49}{\log p} \cdot K. \label{eqn:formula-for-C}
    \end{align}
    for the value of $K$ given above. 

    \begin{remark}
        It follows from a proof analogous to that of \cite[Lemma A.5]{KMRX} that $K$ and $C_p(X)$ are decreasing for $X > 30.$ This fact will be useful when applying Corollary \ref{cor:E(p)-relation}.
    \end{remark}
    
\begin{remark}\label{rmk:domain-restriction-on-X}
    Note that the above value of $C_p(X)$ is only guaranteed to be valid for $X$ large enough such that $X > 30$ and $\frac{2}{3} + \frac{2 \log K}{3 X} - \frac{2\log \left( X \log X \right)}{3X} - \frac{2 \cdot 1.06602}{3 \log X} \in \lrp{\frac{1}{3}, \frac{2}{3}}$. We will only ever consider such $X$, so we take this condition for granted.
\end{remark}

\subsection{Explicit forms of \texorpdfstring{$C_p^{\operatorname{pm}}(X)$}{Cppm(X)} and \texorpdfstring{$C_p^{\operatorname{pp}}(X)$}{Cppp(X)}}\label{subsection:explicit-form-prime-divisors}

First, consider the case of prime level $P$. Then for any $M \in \mathbb{N}$, we have by the proof of \cite[Theorem 26]{MS07} that
\begin{align}
    \left|
\frac{\# \operatorname{eigen}_{S_2(P)}(\mathbf{T}'_p) \cap I}
{\dim S_2(P)}
-
\int_{-2}^2 \chi_I(x)\, d\mu_p(x)
\right| \leq \frac{\| \mu_p\|}{M + 1} + \frac{20 p^{M/2} + 4}{\dim S_2(P)} \label{eqn:MS-theorem-26}
\end{align}
 for $P > 4p^{M}.$ Meanwhile, \cite[Lemma 24]{MS07} gives that
$$\dim S_2(P) \geq \frac{P - 6}{12}.$$
Utilizing these bounds and setting $M = \left\lceil \frac{\log(P/4)}{\log p} - 1 \right\rceil,$ we find that choosing
\begin{align}
    C_p^{\operatorname{pm}}(X) := \| \mu_p \| \frac{X}{X -\log 4} + \frac{120 e^{X/2} + 48}{e^X - 6} \cdot \frac{X}{\log p} \label{eqn:formula-for-C-p}
\end{align}
guarantees that \  $\text{RHS of \eqref{eqn:MS-theorem-26}} \le C_p^\mathrm{pm}(\log P) \frac{\log p}{\log P}$, as desired.

Next, consider the case of prime power level $Q$.
Then for any $M \in \mathbb{N}$, the proof of \cite[(16)]{MS07} gives that
\begin{align}
    \left|
\frac{\# \operatorname{eigen}_{S_2(Q)}(\mathbf{T}'_p) \cap I}
{\dim S_2(Q)}
-
\int_{-2}^2 \chi_I(x)\, d\mu_p(x)
\right| \leq \frac{\| \mu_p\|}{M + 1} + \frac{68 p^{3M/2} + 4\sqrt{Q}}{\dim S_2(Q)}. \label{eqn:MS-equation-16}
\end{align}
It follows as an immediate consequence of the formula \cite[Theorem 13]{MS07} for $\dim S_k(N)$ that
\begin{align}
    \dim S_2(Q) &\geq \frac{Q + P^{n - 1}}{12} - \frac{2}{3} - \frac{1}{2} - \begin{cases}
        P^{(n - 1)/2} & \text{if $n$ is odd}\\
        \frac{1}{2} \left( P^{n/2} + P^{n/2 - 1} \right) & \text{if $n$ is even}
    \end{cases}
    \qquad \text{(where $Q=P^n$)}
    \\
    &\geq \frac{Q}{12} - \frac{5}{12} \sqrt{Q} - \frac{11}{6}\\
    & \geq \frac{Q}{12.00273} \qquad \text{for $Q \geq e^{20}$.}
\end{align}
Therefore, again setting $M = \left\lfloor \frac{c \log Q}{\log p} \right\rfloor,$ we obtain,
\begin{align}
    \text{RHS of \eqref{eqn:MS-equation-16}} &\leq \frac{\log p}{\log Q} \cdot \frac{\| \mu_p\|}{c} + 12.00273 \cdot \frac{68 Q^{3c/2} + 4\sqrt{Q}}{Q}\\
    &\leq \frac{\log p}{\log Q} \left(  \frac{\| \mu_p\|}{c} + \frac{12.00273 \cdot 72}{\log p} \cdot \frac{ Q^{3c/2} \cdot \log Q}{Q} \right) \qquad \text{(when $c \geq \frac{1}{3}$)}\\
    &\leq \frac{\log p}{\log Q}  \left( \| \mu_p \| \left( \frac{2}{3} + \frac{2 \log K}{3X} - \frac{2\log X}{3X}  \right)^{-1} +  \frac{864.197}{\log p}K  \right),
\end{align}
for any $K > 0$, by the substitutions $X = \log N$ and $c = \frac{2}{3} + \frac{2 \log K}{3X} - \frac{2\log X}{3X}.$ By the technique of \cite[Appendix A]{KMRX}, the optimal choice of $K,$ dependent on $N,$ is
\begin{align}
    K := \frac{3\|\mu_p \| \log p}{8 \cdot 864.197} X \cdot  \, W\! \left( \sqrt{ \frac{3 \|\mu_p \| \log p \cdot e^X}{8 \cdot 864.197}  }  \right)^{-2}.
\end{align} 
Thus, we may set
\begin{align}
    C_p^{\operatorname{pp}}(X) := \| \mu_p \| \left( \frac{2}{3} + \frac{2 \log K}{3X} - \frac{2\log X}{3X}  \right)^{-1} +  \frac{864.197}{\log p}K \label{eqn:formula-for-C-pp}
\end{align}
for the above value of $K.$ A similar domain restriction on $X$ to the one discussed in Remark \ref{rmk:domain-restriction-on-X} applies.

\subsection{Explicit forms of \texorpdfstring{$C_p^{\nw}(X)$}{Cp(new)(X)} and \texorpdfstring{$C_p^{\nw, \operatorname{pp}}(X)$}{Cp(new pp)(X)}} \label{subsection:explicit-form-newspace}

Recall the identity \cite[Theorem 13.5.7]{cohen2017modular}
\begin{align}
    \operatorname{Tr}_{S_2^\nw(N)} (\mathbf{T}_m) &= \sum_{N' \mid N} \beta\left( \frac{N}{N'} \right) \operatorname{Tr}_{S_2(N')}(\mathbf{T}_m) \label{eqn:trace-formula-newspace}.
\end{align}
for the multiplicative function $\beta(P^r) := \begin{cases}
        1 & r = 0, 2\\
        -2 & r = 1\\
        0 & r \geq 3
    \end{cases}.$ In the proof of \cite[Theorem 19]{MS07}, the numerator of the second term of \eqref{eqn:Murty-Sinha-Theorem-19} is obtained through a linear combination of the term
\begin{align}
    \begin{cases}
        \left| 
        \operatorname{Tr}_{S_2(N)} \mathbf{T}'_{p^{|m|}} - \operatorname{Tr}_{S_2(N)}\mathbf{T}_{p^{|m| - 2}}' \right| & |m| \geq 2\\
        \left| \operatorname{Tr}_{S_2(N)} \mathbf{T}'_{p} \right| & |m| = 1\\
        0 & m = 0
    \end{cases}
\end{align}
for $|m| \leq M.$ But \eqref{eqn:trace-formula-newspace} gives that
\begin{align}
    \left| 
        \operatorname{Tr}_{S_2^\nw (N)} \mathbf{T}'_{p^{|m|}} - \operatorname{Tr}_{S_2^\nw(N)}\mathbf{T}_{p^{|m| - 2}}' \right| &= \left|  \sum_{N' \mid N} \beta\left( \frac{N}{N'} \right) \left( \operatorname{Tr}_{S_2(N')} \mathbf{T}'_{p^{|m|}} - \operatorname{Tr}_{S_2(N')}\mathbf{T}_{p^{|m| - 2}}' \right) \right|\\
        &\leq \max_{N' \mid N} \left|\operatorname{Tr}_{S_2(N')} \mathbf{T}'_{p^{|m|}} - \operatorname{Tr}_{S_2(N')}\mathbf{T}_{p^{|m| - 2}}' \right| \cdot \sum_{N' \mid N} \left| \beta \left( \frac{N}{N'} \right) \right|\\
        &= \max_{N' \mid N} \left|\operatorname{Tr}_{S_2(N')} \mathbf{T}'_{p^{|m|}} - \operatorname{Tr}_{S_2(N')}\mathbf{T}_{p^{|m| - 2}}' \right| \cdot \prod_{P^r \| N} \left(3 + \mathds{1}_{r \geq 2} \right)\\
        &\leq \max_{N' \mid N} \left|\operatorname{Tr}_{S_2(N')} \mathbf{T}'_{p^{|m|}} - \operatorname{Tr}_{S_2(N')}\mathbf{T}_{p^{|m| - 2}}' \right| \cdot 4^{\omega(N)},
\end{align}
and likewise that $\left| \operatorname{Tr}_{S_2^\nw(N)} \mathbf{T}'_{p} \right| \leq 4^{\omega(N)} \cdot \max_{N' \mid N} \left| \operatorname{Tr}_{S_2(N')} \mathbf{T}'_{p} \right|.$ Therefore, a modification of the proof of \cite[Theorem 19]{MS07} gives that for any $M 
\in \mathbb{N},$
\begin{align}
    &\quad \left|
\frac{\# \operatorname{eigen}_{S_2^\nw(N)}(\mathbf{T}'_p) \cap I}
{\dim S_2^\nw(N)}
-
\int_{-2}^2 \chi_I(x)\, d\mu_p(x)
\right| \\
&\leq \frac{\| \mu_p\|}{M + 1} + 4^{\omega(N)}\cdot  \frac{4p^M 2^{\omega(N)} \sup_{n^2 < 4p^M} \psi(n) + 2\Sigma(N) +  2p^{M/2}}{\dim S_2^\nw(N)}, \label{eqn:Murty-Sinha-Theorem-19-newspace}
\end{align}
where $\Sigma(N)$ is as defined in Appendix \ref{subsection:explicit-form-C}. We can furthermore lower-bound $\dim S_2^\nw(N)$ by
\begin{align}
    \dim S_2^\nw(N) &\geq \frac{\psi^\nw(N)}{12} - \frac{1}{2} \sqrt{N} - \frac{7}{12} 2^{\omega(N)}  
    &\text{(by \cite[Theorem 1]{martin2005dimensions})}\\
    &\geq \frac{N }{12} \cdot \frac{C_{\text{Artin}}}{e^{\gamma} \log \log N - \frac{2.5}{\log \log N}} - \frac{1}{2} \sqrt{N} - \frac{7}{12} 2^{\omega(N)}  
    &\text{(for $N \geq 223092870$, by \eqref{eqn:bound-on-psi-new})}\\
    &\geq \frac{1}{67} \cdot \frac{N}{\log \log N}.
\end{align}
Via the same approach in Appendix \ref{subsection:explicit-form-C}, we obtain that $\eqref{eqn:Murty-Sinha-Theorem-19-newspace} \leq C_p^{\nw}(\log N) \frac{\log p}{\log N},$ where 
\begin{align}
        \quad C_p^\nw(X) :=  \| \mu_p \| \left( \frac{2}{3} + \frac{2\log K}{3X} - \frac{2 \, \log \left( X (\log X)^2 \right)}{3 X} - \frac{2 \cdot 2.98475}{3\log X}  \right)^{-1} + \frac{1516.04}{ \log p} \cdot K   \label{eqn:formula-for-C-new}
    \end{align}
    for
    \begin{align}
        K := \frac{3 \| \mu_p \| \log p}{8 \cdot 1516.04} \, X \cdot W \! \left( \sqrt{\frac{3 \| \mu_p \| \log p}{8 \cdot 1516.04 \, (\log X)^2}} \exp \left( \frac{X}{2} - \frac{2.98475 \, X }{2 \log X} \right) \right)^{-2},
    \end{align}
    with a similar domain restriction on $X$ to the one discussed in Remark \ref{rmk:domain-restriction-on-X}.\\

    Restricting to prime power level $Q,$ we have $\sum_{M \mid Q} \left| \beta \left( \frac{Q}{M} \right) \right| \leq 4$ and
    \begin{align}
    \dim S_2^\nw(Q) &\geq \frac{Q}{12} \prod_{P^r \| Q} \left( 1 - \frac{1}{p} - \frac{\mathds{1}_{r \geq 2}}{p^2} + \frac{\mathds{1}_{r \geq 3}}{p^3}  \right) - \frac{1}{2} \sqrt{Q} - \frac{7}{12} 2^{\omega(Q)} \quad\ \text{(by \cite[Theorem 1]{martin2005dimensions})}\\
    &\geq \frac{Q}{12} \left( 1 - \frac{1}{\sqrt{Q}} - \frac{1}{Q}  \right) - \frac{1}{2} \sqrt{Q} - \frac{7}{6}\\
    &\geq \frac{Q}{12.00382} \qquad \text{for $Q \geq e^{20}$.}
\end{align}
Through a modification of the technique of Appendix \ref{subsection:explicit-form-C}, we can set
\begin{align}
    C_p^{\nw, \operatorname{pp}}(X) := \| \mu_p \| \left( \frac{2}{3} + \frac{2\log K}{3X} - \frac{2 \log X}{3X} \right)^{-1} + \frac{2592.83}{\log p} K,
\end{align}
where
\begin{align}
    K := \frac{3 \| \mu_p \| \log p}{8 \cdot 2592.83} X \cdot W \!\left(  \sqrt{\frac{3 \| \mu_p \| \log p \cdot e^X}{8 \cdot 2592.83}} \right)^{-2},
\end{align}
with a similar domain restriction on $X$ to the one discussed in Remark \ref{rmk:domain-restriction-on-X}.

\subsection{Explicit forms of \texorpdfstring{$C_p^{\operatorname{AL}}(X)$}{Cp(sigma)(X)}, \texorpdfstring{$C_p^{\nw, \operatorname{AL}}(X)$}{Cp(new sigma)(X)}, and \texorpdfstring{$C_p^{\operatorname{AL}, \operatorname{pm}}(X)$}{Cp(sigma, p)(X)}}\label{subsection:explicit-form-sigma}

Bounds for $C_p^{\operatorname{AL}}(N)$ analogous to those derived in the previous sections of this appendix are already present in \cite[\S3.1]{KMRX} and \cite[\S3.3]{KMRX}. These bounds are designed to be uniform over all choices of prime $p$ and weight $k.$ We briefly improve these bounds by making the dependence on $p$ explicit and specializing to weight $k = 2.$

Modifying \cite[\S3.1]{KMRX}, we define the functions
\begin{align}
        V_x(K) &:= \frac{1}{2} - \frac{\log\left( \frac{x^{D + 1}}{K}  \right)}{x} - \frac{\log \left( 2^{1.38407B +  1.53794C}  \right)}{\log x}\\
        F_x(K) &:= \| \mu_p \| \cdot V_x(K)^{-1} + \frac{A}{\log p} K.\\
            K(x) &:= \frac{\| \mu_p \|}{4A} \cdot W\! \left( \sqrt{\frac{\| \mu_p \|}{4A}} \cdot x^{-D/2} \exp\left( \frac{x}{4} - \frac{x \log\left( 2^{1.38407B +  1.53794C}  \right)}{2\log x}  \right)  \right)^{-2},
\end{align}
where $A > 4 \| \mu_p \|$ and $B, C, D \in \mathbb{Z}_{\geq 0}$ are fixed constants. Then a slight modification of \cite[Lemma 3.4]{KMRX}, with identical proof, gives the following.

\begin{lemma}\label{lem:KMRX-lem-3.4-substitute}
    Fix $A > 4 \| \mu_p \|$ and $B, C, D \in \mathbb{Z}_{\geq 0}.$ Suppose that $S(N, p)$ is a function over integers $k, p \geq 2$ and $N \geq 1129$ such that for any $c \in \left( 0, \frac{1}{2} \right),$
    $$S(N, p) \leq \frac{\| \mu_p \|}{c} \cdot \frac{\log p}{\log N} + A \cdot \frac{2^{ B\omega(N)} \sigma_0(N)^C \log(N)^D}{N^{1/2 - c}}.$$
    Furthermore, suppose that $X > 1$ such that for all $x \geq X,$
    \begin{equation}
            \log K(x) < (D + 1) \left( \log x - 1 \right) + \frac{x \log\left( 2^{1.38407B +  1.53794C} \right)}{\left( \log x \right)^2}.
        \end{equation}
    Then, for all $N \geq e^X,$ we have $S(N, p) \leq \frac{\log p}{\log N} \cdot F_X(K(X)).$
\end{lemma}

Then, a modification of \cite[Lemma 3.5]{KMRX} using Lemma \ref{lem:KMRX-lem-3.4-substitute} in place of \cite[Lemma 3.4]{KMRX} gives that for all $N \geq e^X \geq e^{593.590}$,
\begin{align}
    \left| \frac{\# \operatorname{eigen}_{S_2^\sigma(N)}(\mathbf{T}'_p) \cap I}{\dim S_2^\sigma(N)} - \int_{-2}^2 \chi_I (x) \, d \mu_p(x)  \right| \leq F_{X}(K(X)) \cdot \frac{\log p}{\log N}, \label{eqn:KMRX-lem-3.5-substitute}
\end{align}
with the fixed constants $A = 2870.83, \ B = D = 1,$ and $C = 2.$ Then we can set $C_p^{\operatorname{AL}}(X) := F_X(K(X))$ for these choices of $A,B,C,D$.
Note that the constant $2870.83$ in the bound \eqref{eqn:KMRX-lem-3.5-substitute} is obtained by multiplying the constant $4550.16$ from \cite[Lemma 3.5]{KMRX} by $\frac{\log 2}{\log 3}$. 

Similar explicit forms for $C_p^{\nw, \operatorname{AL}}$ and $C_p^{\operatorname{AL}, \operatorname{pm}}$ can be obtained immediately by modifying \cite[Lemma 3.6]{KMRX} and \cite[Lemma 3.8]{KMRX} analogously.

\section*{Acknowledgments}

The authors would like to acknowledge Peter Sarnak and Will Sawin for helpful comments and suggestions. This work was supported by National Science Foundation grant DMS-2349174. Hui Xue is supported by Simons Foundation grant MPS-TSM-00007911.

\bibliographystyle{plain}
\bibliography{ss-biblio.bib}

\end{document}